\documentclass[a4paper,12pt]{article}
\usepackage{amsmath}
\usepackage{amssymb}
\usepackage{amsfonts,times}
\usepackage{theorem}
\usepackage{hyperref}
\AtBeginDocument{}
\newtheorem{theorem}{Theorem}

\newtheorem{corollary}[theorem]{Corollary}

\newtheorem{definition}[theorem]{Definition}

\newtheorem{lemma}[theorem]{Lemma}

\newtheorem{convention}[theorem]{Convention}

\newtheorem{proposition}[theorem]{Proposition}
{\theorembodyfont{\upshape}\newtheorem{remark}[theorem]{Remark}}
{\theorembodyfont{\upshape}}

\newenvironment{proof}[1][Proof]{\noindent\textbf{#1.} }{\ \hfill \rule{0.5em}{0.5em}}

\newcommand{\cF}{\mathcal{F}}
\newcommand{\cM}{\mathcal{M}}
\newcommand{\cT}{\mathcal{T}}
\newcommand{\cE}{\mathcal{E}}

\newcommand{\IN}{\mathbb{N}}
\newcommand{\IE}{\mathbb{E}}
\newcommand{\IP}{\mathbb{P}}
\newcommand{\IR}{\mathbb{R}}
\newcommand{\IV}{\mathbb{V}}

\newcommand{\neu}{}

\newcommand{\supp}{\text{\rm{supp}}}
\begin{document}

\title{Random matrices with exchangeable entries}
\author{Werner Kirsch
\\
%EndAName
Fakult\"{a}t f\"{u}r Mathematik und Informatik\\
FernUniversit\"{a}t in Hagen, Germany \and Thomas Kriecherbauer \\
%EndAName
Mathematisches Institut\\
Universit\"{a}t Bayreuth, Germany}
\date{}

\maketitle
%\marginpar{Before submission we need to find out what the "local law" people have to say about all this for Wigner band matrices, centred or not}
\begin{abstract}
We consider ensembles of real symmetric band matrices with
entries drawn from an infinite sequence of exchangeable random variables,
as far as the symmetry of the matrices permits. In general the entries of the upper
triangular parts of these matrices are correlated and no smallness or sparseness of these
correlations is assumed. It is shown that the eigenvalue
distribution measures still converge to a semicircle but with random scaling.
We also investigate the asymptotic behavior of the corresponding $\ell_2$-operator norms.
The key to our analysis is a generalisation of a classic result by de Finetti that allows
to represent the underlying probability spaces as averages of Wigner band ensembles
with entries that are not necessarily centred. Some of our results appear to be new
even for such Wigner band matrices.
\end{abstract}

\section{Introduction}
In this paper we consider (full) real symmetric random matrices of the form

\begin{equation}
X_N =
\begin{pmatrix}
 X_N (1,1) & X_N (1,2) & \cdots & X_N (1,N)\\
 X_N (2,1) & X_N (2,2) & \cdots & X_N (2,N)\\
 \vdots    & \vdots    &        & \vdots   \\
 X_N (N,1) & X_N (N,2) & \cdots & X_N (N,N)
\end{pmatrix}
\end{equation}
for certain random schemes $X_{N}(i,j)$ (with $X_{N}(i,j)=X_{N}(j,i) \in \mathbb{R}$),
as well as real symmetric band random matrices
where $X_{N}(i,j)$ is random in a strip of size $w_{N}\to\infty$
centred around the diagonal $i=j$ and $X_{N}(i,j)\equiv 0$ otherwise.
We use the word {\em full}  to distinguish from the case of
band matrices. The parantheses above indicate that we will omit this specification if there is no danger
of confusion.

Let us denote the underlying probability space by $(\Omega,\cF,\IP)$ and the expectation with
respect to $\IP$ by $\IE$.

For any symmetric $N \times N$-matrix $M$ we denote
the eigenvalues of $M$ by $\lambda_j (M)$. We order these eigenvalues such that
\begin{equation*}
\lambda_1 (M) \leq \lambda_2 (M) \leq \cdots \leq \lambda_N (M)
\end{equation*}
where degenerate eigenvalues are repeated according to their multiplicity.

The eigenvalue distribution measure $\nu_M$ of $M$ is defined by
\begin{align*}
\nu_M (A) & = \frac{1}{N} \vert \{j \vert \lambda_j (M) \in A \} \vert\\
& = \frac{1}{N} \sum_{j=1}^N \delta_{\lambda_j (M)} (A)
\end{align*}
where $\vert B \vert$ denotes the number of points in $B, N$ - as above -
is the dimension of the matrix $M$, $A$ is a Borel-subset of $\IR$ and $\delta_a$
is the Dirac measure in $a$, i.e.
\begin{equation*}
\delta_a (A) =
\begin{cases}
1 & \text{~ if ~} a \in A\\
0 & \text{~  otherwise} \,.
\end{cases}
\end{equation*}

In this paper we study the limiting behavior of both the eigenvalue distribution measures and the
$\ell_2$-operator norms as the matrix dimension $N$ becomes large. Within the theory of random
matrices the first results on these quantities were obtained for full {\em Wigner ensembles}
(see~Theorem~\ref{theo:classicalWigner} below) using the method of moments.

The distinctive feature of Wigner ensembles is that
for each fixed matrix size $N$ the entries $X_{N}(i,j)$, $1\leq i \leq  j\leq N$, of the upper triangular part
are independent, \neu identically distributed random variables. For this class of matrix ensembles it has been shown in great generality
that the eigenvalue distribution measures converge to the famous semicircle law. Having obtained such
a universal limiting law it is natural to test its range of validity. For example, one might ask whether the
assumption of independence of the matrix entries in the upper triangular part can be relaxed. Indeed,
a number of matrix ensembles with correlated entries have been introduced in the literature
and their limiting spectral distributions have been analysed. We refer the reader to the survey
\cite{KK1} for a detailed description of these results \neu(see also the paper \cite{Erdos18a} for recent developments) . Most of these ensembles are defined with some
kind of smallness of the correlations built in: They are sparse or they decay. The decay can be with respect
to the distance of the corresponding matrix entries or with respect to the matrix dimension $N$.

The main focus of our paper lies on matrix ensembles with entries that are drawn from
an exchangeable sequence of random variables (see Definition~\ref{def:exchangeable}).
For such models the correlations may neither
be sparse nor decaying. In addition, we do not only consider full matrices but also band
random matrices. One of our main results is Theorem~\ref{thm:DFB-semicircle} where we
show for a large class of such ensembles that the eigenvalue distribution measures still converge
to a semicircle, but its radius may now be random.

A key element in our proof is the
fact that an exchangeable sequence of real-valued random variables can be represented as an average
of i.i.d.~sequences. This classic result is due to de Finetti  \cite{deFinetti, deFinetti2} in the special case
of spin random variables that only assume values $\pm 1$ and was later generalized by
Hewitt-Savage \cite[Theorem 7.4]{HewittSavage} to a setting that includes in particular real-valued
variables. Therefore we can relate matrix ensembles with exchangeable entries to ensembles
with i.i.d~entries and this brings us back to the realm of Wigner ensembles.
Accordingly, we begin the more precise discussion of our results by a definition of Wigner ensembles
that is suitable for the analysis of ensembles with exchangeable entries.

\begin{definition}\label{def:fullWE}
By a {\em (full) Wigner ensemble} we understand a probability measure on sequences $(X_N)_N$
of real symmetric $N \times N$ matrices $X_N$ such that for each fixed $N$
the random variables $X_{N}(i,j)$, $i \leq j$, are independent. Moreover, we require that \neu
the $X_{N}(i,j) $ for all $N,i,j$ have a common distribution $\rho$ with finite moments of all orders.
We call $m=\int x d\rho$ the mean
and $v=\int x^2 d\rho(x)-m^2$ the variance of the Wigner ensemble.
In case the mean vanishes, $m=0$, we say that the Wigner ensemble is {\em centred}.
\end{definition}

A few remarks are in order. First, note that no assumptions are made on
how the entries of $X_N$ and $X_M$ are correlated for $N \neq M$. The reason is
that these correlations play no role for results on spectral limits relevant for this paper.
Secondly, the assumption of identically distributed entries is often relaxed for Wigner ensembles
by conditions that only require agreement of \neu\emph{some} moments. As explained above ensembles with exchangeable entries
are related to the i.i.d.~case and therefore we do not strive for more generality in this respect.
Thirdly, the condition that all moments of the law $\rho$ exist could be downgraded as well.
For Wigner ensembles it is well known how to adapt the arguments in the situation that only a
few moments exist by a truncation procedure. In order to avoid the associated substantial technicalities
we restrict ourselves to the case that all moments exist. Fourthly, and this is the most important point,
we do {\em not} require that the entries are centred random variables. One motivation for this is
our recent work on random matrices with Curie-Weiss distributed entries \cite{KK2} where non-zero means
are generated by magnetisation at low temperatures. The results we obtain for means $m \neq 0$
appear to be new, some of them even in the case of Wigner band ensembles (see Definition~\ref{def:eeebm}).
In comparison, the influence of the variance $v$ on the spectrum is simple,
because it translates to a linear scaling by the factor $\sqrt{v}$.

Based on the work of Wigner
and others (see \cite{Wigner1, Wigner2}, \cite{Grenander}, \cite{Arnold})
it is well known that for centred Wigner ensembles
$X_{N}$
the eigenvalue distribution measures $\mu_{N}$ of the matrix $\frac{1}{\sqrt{N}}X_{N}$ converge
in the case $v>0$ to the (scaled) \emph{semicircle distribution}, i.\,e.~to the
measure $\sigma_{v}$ with density (with respect to the Lebesgue measure)
\begin{equation}\label{eq:defscd}
s_{v} (x) ~=
~\frac{1}{2 \pi v} \sqrt{(4 v - x^2)_{+}}\,.
\end{equation}
Above we use the notation $g_{+}(x):=\max\big(g(x),0\big)$. The classical Wigner case corresponds to
$v=1$, the above slightly more general case follows through scaling.

For our purposes it will be convenient to include the trivial case of variance $v=0$.
In the centred case the entries of the Wigner matrices are then equal to zero almost surely.
Thus the corresponding eigenvalue
distribution measures of the matrices $\frac{1}{\sqrt{N}}X_{N}$ are all given by the Dirac measure
$\delta_0$. We therefore extend definition \eqref{eq:defscd} for $\sigma_{v}$ by
\begin{equation}\label{eq:defscd_v=0}
\sigma_{0} ~:= \delta_{0}\,.
\end{equation}
Despite its degeneracy we call $\delta_{0}$ a semicircle distribution throughout the paper.

There are various forms of convergence for sequences of random measures:
\begin{definition}\label{def:Tprobwconv}
Let $(\Omega, \mathcal{F}, \IP)$ be a probability space and let
$\mu_N^\omega$ and $\mu^\omega$, $\omega \in \Omega$, be random probability measures on $(\IR, B(\IR))$.
\begin{enumerate}
\item[1)] We say that $\mu_N^\omega$ converges to $\mu^\omega$ \emph{weakly in expectation},
if for every

$f \in C_b (\IR)$, the set of bounded continuous functions on $\IR$\,,
\begin{equation}\label{eq:convwe}
\IE \Big( \int f (x) ~  d \mu_N^{\omega} (x) \Big) \rightarrow \IE \Big(\int f (x) ~  d \mu^\omega (x) \Big)
\end{equation}

as $N \rightarrow \infty$.
\item[2)] We say that $\mu_N^\omega$ converges to $\mu^\omega$ \emph{weakly in probability},
if for every $f \in C_b (\IR)$ and any $\epsilon > 0$
\begin{equation*}
\IP \Big(\;\Big| \int f (x) ~ d\mu_N^\omega (x) - \int f (x) ~ d \mu^\omega (x) \Big| > \epsilon \Big) \rightarrow 0
\end{equation*}
as $N \rightarrow \infty$.
\item[3)] We say that $\mu_N^\omega$ converges to $\mu^\omega$
\emph{weakly $\IP$-almost surely} if there is a set $\Omega_0 \subset \Omega$
with $\IP (\Omega_0) = 1$ such that $\mu_N^\omega
 \Rightarrow \mu^\omega$ for all $\omega \in \Omega_0$. Here $\mu_N^\omega
 \Rightarrow \mu^\omega$ means weak convergence, i.e.~for every $f \in C_b (\IR)$
 \begin{equation*}
 \int f (x) ~ d\mu_N^\omega (x) \rightarrow \int f (x) ~ d \mu^\omega (x) \quad \mbox{ as } N \rightarrow \infty\,.
\end{equation*}
\end{enumerate}
%NACHHILFEBEDARF FUER THOMAS: IST 3) AEQUIVALENT DAZU, DASS FUER JEDES $f$
%FAST SICHERE KONVERGENZ DER INTEGRALE GILT, WAS AUF DEN ERSTEN BLICK SCHWAECHER
%ERSCHEINT?

%FALLS DIES NICHT DER FALL SEIN SOLLTE, MUESSTEN WIR ABKLOPFEN, OB
%THEOREME \ref{theo:classicalWigner}a) UND SEIN ANALOGON FUER BANDMATRIZEN, UND VOR ALLEM AUCH
%DIE ARGUMENTE AM ENDE VON ABSCHNITT 2 MIT DIESEM STRENGEREN KONVERGENZBEGRIFF KOMPATIBEL SIND!

%FALLS WIR DIE DEFINITION 3) AENDERN MUESSTEN; DANN MUESSTEN WIR DEN BEWEIS DES
%HAUPTRESULTATS VON ABSCHNITT 3.2 ANPASSEN UND DIE FORMULIERUNG DESSEN KOROLLARS.
\end{definition}

We can now formulate: For centred Wigner matrices $X_{N}$ with variance $v$
the eigenvalue distribution measures $\mu^{\omega}_N$
of $\frac{1}{\sqrt{N}}X_{N}$ converge weakly $\IP$-almost surely to $\sigma_{v}$
\cite{Arnold}.
%\marginpar{Zitat ok?}.

Besides the limiting spectral distribution we also want to understand the behavior of
the $\ell_2$-operator norms $||X_N||_{op}$ as the matrix size $N$ becomes large.
Recall that for real symmetric matrices the operator norm is given by the largest eigenvalue
in modulus,
$$||X_N||_{op}= \mbox{ max}(|\lambda_1(X_N)|,|\lambda_N(X_N)|).$$
Guided by the semicircle law one might expect that $||X_N||_{op} \sim 2\sqrt{vN}$. Indeed,
the semicircle law can be used to show that the limes inferior of $||X_N||_{op}/\sqrt{N}$
is bounded below by $2\sqrt{v}$ (see proof of Parts~II and III of Theorem \ref{theorem:W1}).
However, an upper bound for the operator norm cannot possibly be extracted from the
semicircle law since, for example, a single outlier of the spectrum has no effect on the limiting
spectral distribution but may determine the operator norm. Therefore additional arguments are needed.
These were provided in \cite{BaiY}, see also \cite{FuerediK},
%\marginpar{Zitate ok?}
where it was shown for centred Wigner ensembles that
$||X_N||_{op}/\sqrt{N}$ converges to $2\sqrt{v}$ almost surely.

The classical results that we have discussed so far can be summarized as follows.
\begin{theorem}
\label{theo:classicalWigner}
Let $(X_N)_N$ be a (full) centred Wigner ensemble with variance $v$ in the sense of
Definition \ref{def:fullWE}. Then:
\begin{itemize}
\item[a)] The eigenvalue distribution measures $\mu^{\omega}_N$
of $\frac{1}{\sqrt{N}}X_{N}$ converge weakly $\IP$-almost surely to $\sigma_{v}$\,.
\item[b)] The scaled operator norms $||X_N||_{op}/\sqrt{N}$ converge $\IP$-almost surely to
$2\sqrt{v}$\,.
\end{itemize}
\end{theorem}

Our main goal is to generalise Theorem~\ref{theo:classicalWigner} to ensembles
of full or banded matrices with exchangeable entries.
In the remainder of the Introduction we outline the plan of the paper highlighting our main results
along the way.

As a first step we formulate
in Section \ref{sec:WE} the semicircle law  for Wigner band matrices (see
Definition~\ref{def:eeebm}). The corresponding result is stated in Theorem~\ref{thm:wignerband}. Note that
the scaling $X_{N}/\sqrt{N}$ of Theorem~\ref{theo:classicalWigner}
needs to be replaced by $X_{N}/\sqrt{w_N}$ where $w_N$ is a measure for the bandwidth.
The case of centred entries is essentially known \cite{BMP, MPK} (see also \cite{Catalano, Fleermann})
and this is the starting point for our proof.
In order to analyse arbitrary means $m$ we write
\begin{equation}\label{decomposition_one}
\frac{1}{\sqrt{w_N}} X_N= \frac{1}{\sqrt{w_N}}(X_N - \IE(X_N)) + \frac{1}{\sqrt{w_N}}\IE(X_N)\,.
\end{equation}
The first summand
has centred entries and its eigenvalue distribution measure therefore obeys the semicircle law.
The main work in the proof consists of showing that the deterministic matrix
$\IE(X_N)/\sqrt{w_N}$ can be decomposed into two parts such that one of them has small enough
norm and the other one has small enough rank to allow the semicircle law to persist.

As can be seen from the statement of Theorem~\ref{thm:wignerband} the result for Wigner band
matrices is a little more involved than for full matrices.
For example, even in the case of centred entries one needs (at least for our proof) an additional (mild) condition on the bandwidths
to improve from convergence in probability to almost sure convergence. This subtlety and its
proof seem to be somewhat buried in the literature.
We refer the reader to \cite{Fleermann} for a proof in a more general setting that also includes the case of correlated entries.
Band matrices with linearly growing bandwidths appear as a special case of `general Wigner-type matrices' in \cite{Erdos17}, where a `local
law' for eigenvalue statistics is proved. In particular, their results imply almost sure convergence of the eigenvalue distribution measures for these band matrices.

For the convenience of the reader we sketch a proof of almost sure convergence
for centred Wigner band ensembles for our simpler situation in Subsection \ref{app:sec1} of the Appendix.

Section \ref{sec:DFE} is devoted to generalising Theorem~\ref{thm:wignerband} to band
ensembles with exchangeable entries. We call them de Finetti band ensembles in reference to the remarkable work of de Finetti
\cite{deFinetti, deFinetti2} on which our analysis is based (see Definition~\ref{def:dfbm}).
Observe that we use a definition of band matrices that includes full matrices by choosing the bandwidth sufficiently large.
Theorem~\ref{thm:DFB-semicircle} states our main result for these ensembles. As it was already mentioned above, the only difference
between the results in the Wigner and in the de Finetti case is that the limiting law is given by the semicircle $\sigma_V$
rather than $\sigma_v$. I.e.~the variance $v$ of the Wigner ensemble needs to be replaced by a
real-valued random variable $V$ that we call the limiting empirical variance of the ensemble (see Definition~\ref{def:empirical}).

We show in Subsection \ref{ssec:3.3} that a non-random limit law for the eigenvalue distribution measures can be achieved
but, except for trivial cases, one needs to settle for the weaker notion of convergence in expectation
(cf. Definition~\ref{def:Tprobwconv}).
In addition we derive some properties of the deterministic limit law including
a characterisation of all cases in which it is a semicircle.

So far we have generalised part a) of Theorem~\ref{theo:classicalWigner}. In the final section
of this paper we study the corresponding operator norms. As explained above the statement of
Theorem~\ref{theo:classicalWigner} the main task is to obtain upper bounds once the limit law
for the eigenvalue distribution measures is established.
Observe that for centred Wigner band matrices it was shown in \cite{BMP}
%\marginpar{Zitat ok}
that  $||X_N||_{op}/\sqrt{w_N}$ is unbounded if the bandwidth grows
slowly enough with matrix dimension $N$.
%\marginpar{Verweis auf Remark~\ref{remark:Pastur}?}
With Theorem~\ref{theorem:W1} we provide a result in the opposite direction. We prove for
centred Wigner band ensembles with bandwidths~$w_N$
growing at least of order $N^{\alpha}$ for some arbitrarily small $\alpha > 0$ that $||X_N||_{op}/\sqrt{w_N}$
converges to $2\sqrt{v}$ almost surely in all cases where the eigenvalue distribution measures converge
to $\sigma_v$.

The proof of this result uses the strategy that was introduced in \cite{BaiY}.
We follow the presentation of the monograph \cite{Tao}.
%\marginpar{Zitate ok?}
Both references deal with ensembles of full matrices.
Although the generalisation to band matrices does not pose any difficulties we provide
a proof in the second subsection of the Appendix. The reason is that we have improved on some of the
inequalities (see in particular Lemma~\ref{lemma:MB}) in order to obtain weaker conditions on the required rate of growth
for the bandwidths (see Remark~\ref{remark:Pastur} and Lemma~\ref{lemma:fundamental_estimate}), an issue
that is not present in the case of full random matrices.

For `general Wigner-type matrices' upper bounds for the operator norms are given in \cite{Erdos18}.
As mentioned above these matrices include band matrix ensembles with linearly growing bandwidths.

Finally we consider operator norms for de Finetti band ensembles and for Wigner band ensembles that are not
centred in Subsection \ref{ssec:4.2}. This is a much less subtle question than in the centred case since the deterministic
part $\IE(X_N)$ in the decomposition~\eqref{decomposition_one} has operator norm of order $w_N$
(see Lemma~\ref{lemma:G1}) that dominates the centred part that is only of order $\sqrt{w_N}$.
Therefore the mean of the entries of the Wigner ensemble and the (random) empirical mean of the de Finetti ensemble
(see Definition~\ref{def:empirical}) respectively determine the asymptotic behavior of the operator norms.
In the special case of full matrices one may use that the matrix of means $\IE(X_N)$ has rank one to show that the
discrepancy between the $N$-scaling of the operator norms and the $\sqrt{N}$-scaling of the eigenvalue
distribution measures is caused by a single outlier of the spectrum, see Proposition~\ref{prop:conv_op'} and part 2 of Remark~\ref{rem:F1}.

\medskip \textbf{Acknowledgement}
Most of this work was done during the first author's visit to the Institute for Mathematics of the University of Bayreuth,
and the second author's stay at the Lehrgebiet Stochastics at the FernUniversit\"{a}t in Hagen.
We thank those institutions for their support and their great hospitality. We would also like to thank Michael Fleermann for
many valuable discussions.

%In this paper we consider random matrices with entries which form an exchangeable
%scheme of random variables. More precisely: There is an infinite scheme $X(i,j), i\leq j$ of
%exchangeable random variables such that $X_{N}(i,j)=X(i,j)$ for $1\leq i \leq j\leq N$
%and $X_{N}(i,j)=X_{N}(j,i)$ for $i>j$. Both for the case of full random matrices and
%for band matrices we prove an appropriate
%counterpart to the semicircle law. \marginpar{Neu schreiben!}
%As a rule, the limit measure is no longer a semicircle
%distribution but a mixture of such distribution with different scaling. Moreover, the convergence
%is only weakly in expectation. We characterize those cases, for which the limit \emph{is}
%a semicircle distribution. For such ensembles the convergence turns out to be weakly in probability.

%In the papers \cite{HKW} and \cite{KK2} the semicircle law was proved for Curie-Weiss ensembles
%for arbitrary inverse temperature $\beta\geq 0$. In this paper we consider band random matrices
%of Curie-Weiss type. For the Curie-Weiss ensemble the random variables
%$X_{N}(i,j),1\leq i \leq j\leq N $ form an exchangeable scheme, however in contrast to the previously
%discussed case they are not the restriction of an infinite scheme, in fact, the distribution of
%the $X_{N}(i,j)$ depends on $N$.

\section{Wigner ensembles}
\label{sec:WE}

This section is mainly concerned with the fate of the semicircle law as stated in
Theorem~\ref{theo:classicalWigner}a) if we consider Wigner band matrices
with entries that are not necessarily centred. Our main result in this direction
is Theorem \ref{thm:wignerband} in the third subsection. A precise definition of
Wigner band ensembles is provided in Subsection \ref{ssec:2.2}. There we will
distinguish two different types of band matrices, strict and periodic, that differ
in the way distance is measured on the index set $\{1, \ldots, N\}$. We begin
our discussion with the case of full Wigner matrices and show what happens to
both statements a) and b) of Theorem~\ref{theo:classicalWigner} if one removes the condition
of centred entries.

\subsection{Warm up: Full matrices}

It was already observed by F\H{u}redi and Koml\'{o}s in \cite{FuerediK} that the semicircle
law still holds for full Wigner ensembles with arbitrary means $m$.
\begin{proposition}\label{prop:convrho}
   Let $(X_N)_N$ be a Wigner ensemble with arbitrary mean $m$ and variance $v$ (see
Definition \ref{def:fullWE}). Then the eigenvalue distribution measures $\mu_N$
of $\frac{1}{\sqrt{N}}X_{N}$ converge weakly $\IP$-almost surely to $\sigma_{v}$\,.
\end{proposition}

\begin{proof}
   We give only a brief sketch here since we provide a detailed argument in the more general
   situation of band matrices in the proof of Theorem \ref{thm:wignerband} below.

   The crucial observation is the following. Denote by
   $\cE_{N}$ the $N\times N$-matrix with $\cE_N(i,j)=1$ for all $1\leq i,j\leq N$.
   %Then split the matrix $X_{N}$ as:
   Use
   \begin{align}\label{eq:eN}
      X_{N}~=~\big(X_{N}~-~m\cdot\cE_{N}\big)~+~m\cdot\cE_{N}
   \end{align}
   and observe that the matrix $\cE_{N}$ is a matrix of
   rank one, so it can change the number $|\{j|\lambda_{j}\in A\}|$ of eigenvalues of
   $\frac{1}{\sqrt{N}}(X_{N}-m\cE_{N})$ inside any interval~$A$ by at most $2$, which is negligible
   for the limiting empirical eigenvalue measure (see also Proposition \ref{prop:app2.1}).

   On the other hand the matrix $X_{N}-m\cE_{N}$ is a centred Wigner matrix with variance $v$,
   hence the corresponding empirical eigenvalue measures converge weakly almost surely to the
   semicircle distribution $\sigma_{v}$ by Theorem~\ref{theo:classicalWigner}~a).
\end{proof}

We have just argued that the limiting empirical eigenvalue measure is insensitive to perturbations of rank 1.
The operator norm, however, may feel such a perturbation since it could create a single
outlier of the spectrum. This happens e.g.~in the situation of Proposition \ref{prop:convrho}
if the mean $m$ of the entries does not vanish. In fact, the operator norm of the matrix
$\cE_{N}$ equals $N$ so that $||m \cE_{N}||_{op} = |m|  \, N$. Moreover, it follows
from Theorem~\ref{theo:classicalWigner}b) that $||X_{N}~-~m\cdot\cE_{N}||_{op}$
is of the order $2 \sqrt{v N}$ for large values of $N$. It is therefore asymptotically
negligible when compared to
$||m \cE_{N}||_{op}$. As
\begin{equation*}
||m \cE_{N}||_{op} - ||X_{N}~-~m\cdot\cE_{N}||_{op} \leq
||X_{N}||_{op} \leq
||m \cE_{N}||_{op} + ||X_{N}~-~m\cdot\cE_{N}||_{op}
\end{equation*}
we have proved
\begin{proposition}\label{prop:conv_op}
 For Wigner ensembles $(X_N)_N$ with arbitrary mean
 $m$ and variance $v$ (see Definition \ref{def:fullWE}) the operator norm
 $||X_{N}||_{op}$ satisfies
\begin{equation*}
\mathbb{P}\Big(\lim_{N \to \infty}\frac{||X_{N}||_{op}}{N} = |m| \Big) = 1\,.
\end{equation*}
\end{proposition}

We repeat: In the case that the mean of the entries $m$ does not vanish there
is a discrepancy between the $N$ scaling of the operator norm and the $\sqrt{N}$ scaling
of the semicircle law.

We now formulate the fact that this discrepancy is due to only one outlier.
To this end we introduce the singular values of $X_N$.
By spectral calculus the singular values $s_j (X_N)$ of $X_N$ are
given by the absolute values of the eigenvalues. We order them according to size taking their multiplicities into account
\begin{equation*}
0 \leq s_1 (X_N) \leq s_2 (X_N) \leq \cdots \leq s_N (X_N).
\end{equation*}
The largest singular value $s_N(X_N) = \max \{|\lambda_1(M)|, |\lambda_N(M)|\}$ is of particular interest as it agrees with
the $\ell_2$-operator norm $||X_N||_{op}$.
%\begin{equation*}\label{eq:th1}
%||M||_{op} := \sup \{ |M v|_{\ell_2} \colon v \in \mathbb R^N \mbox{ with } |v|_{\ell_2} \leq 1 \}\,.
%\end{equation*}
We can be sure that $X_N$ has only one outlier of order $N$ if we can prove that
the second largest singular value $s_{N-1}(X_N)$ is of order $\sqrt{N}$. Let us introduce the
notation
\begin{equation*}
||X_N||_{op}^{'} := s_{N-1}(X_N) \,.
\end{equation*}

\begin{proposition}\label{prop:conv_op'}
For Wigner ensembles $(X_N)_N$ with arbitrary  mean
 $m$ and variance $v$ the second largest singular value
 $||X_{N}||_{op}^{'}$ satisfies
\begin{equation*}
\mathbb{P}\Big(\lim_{N \to \infty}\frac{||X_{N}||_{op}^{'}}{\sqrt{N}} = 2 \sqrt{v} \Big) = 1\,.
\end{equation*}
\end{proposition}
\begin{proof}
Let us first consider the case where the mean $m$ of the Wigner ensemble is non-negative.
As the matrix $\cE_{N}$ has rank $1$ and is positive definite the eigenvalues of $X_N$
interlace with the eigenvalues
of $X_N - m \cE_{N}$ in the following way.
\begin{equation}\label{interlace}
\lambda_j (X_N - m \cE_{N}) \leq \lambda_j (X_N) \leq \lambda_{j+1} (X_N - m \cE_{N})
\end{equation}
for all $1 \leq j \leq N-1$ and the first inequality also holds true for $j=N$. The first step in obtaining
upper and lower bounds for the second largest singular value of $X_N$ is the observation that
\begin{equation*}
 \max \{|\lambda_2(X_N)|, |\lambda_{N-1}(X_N)|\} \leq
||X_N||_{op}^{'} \leq \max \{|\lambda_1(X_N)|, |\lambda_{N-1}(X_N)|\}
\end{equation*}
Using the interlacing inequalities \eqref{interlace} we conclude further
\begin{eqnarray*}
\max \{|\lambda_1(X_N)|, |\lambda_{N-1}(X_N)|\} &\leq&
\max \{|\lambda_1(X_N - m \cE_{N})|, |\lambda_{N}(X_N - m \cE_{N})|\}\,, \\
\max \{|\lambda_2(X_N)|, |\lambda_{N-1}(X_N)|\} &\geq&
\max \{|\lambda_3(X_N - m \cE_{N})|, |\lambda_{N-1}(X_N - m \cE_{N})|\}\,.
\end{eqnarray*}
A moment's thought then yields
\begin{equation}\label{estimate:s(N-1)}
s_{N-3}(X_N - m \cE_{N}) \leq
||X_N||_{op}^{'} \leq ||X_N - m \cE_{N}||_{op}\,.
\end{equation}
A similar reasoning shows that the estimates \eqref{estimate:s(N-1)} are also valid
in the case of negative means $m$.
Clearly $(X_N - m \cE_{N})_N$ is a centred Wigner ensemble with variance $v$ so that
Theorem~\ref{theo:classicalWigner}b) implies that the upper bound, divided by
$\sqrt{N}$, converges to $2\sqrt{v}$ almost surely.
As we show in Corollary \ref{cor:s(N-m)} below (see also the discussion above the
statement of Corollary \ref{cor:s(N-m)})
it is also true that the lower bound, divided by
$\sqrt{N}$, converges to $2\sqrt{v}$ almost surely. This completes the proof
up to the verification of Corollary \ref{cor:s(N-m)}.
\end{proof}

\subsection{Strict and periodic band matrices}\label{ssec:2.2}

Let us first define the notion of strict and periodic Wigner band matrices.

\begin{definition}\label{def:eeebm}
\begin{itemize}
\item[a)]
Denote by $(b_N)_N$ a sequence of integers that is bounded by $0 \leq  b_N \leq N-1$.
Then the deterministic prototypes of $N \times N$ strict and periodic band matrices,
$B_N$ and $P_N$, are defined by
\begin{eqnarray*}\label{eq:G1}
B_N(i,j) &:=&
\begin{cases}
        1 &, \text{~  if ~}  |i-j| \leq b_N\\
         0 &, \text{~  if ~}  |i-j| > b_N
         \end{cases} \,,
\\
P_N(i,j) &:=&
\begin{cases}
        1 &, \text{~  if ~}  |i-j|_N \leq b_N\\
         0 &, \text{~  if ~}  |i-j|_N > b_N
         \end{cases} \,,
\end{eqnarray*}
where $|i-j|_N$ denotes the distance between $i$ and $j$ on the circle $\mathbb{Z} / N \mathbb{Z}$,
i.e.~$|i|_N:= \min(|i|, N-|i|)$ for $|i| \leq N$.
\item[b)]
An ensemble of a familiy of $N \times N$ real symmetric matrices
$(W_N)_N$ / $(W^{\text{per}}_N)_N$  is called a {\em strict Wigner band ensemble}/ a
{\em periodic Wigner band ensemble} with mean $m$ and variance $v$,
if it can be generated from a Wigner ensemble $(X_N)_N$ with mean $m$ and variance $v$
(see Definition \ref{def:fullWE}) via
\begin{eqnarray*}
W_N(i,j) &=&
\begin{cases}
        X_N(i,j) &, \text{~  if ~}  B_N(i,j) = 1\\
         0 &, \text{~  else ~}
         \end{cases} \,,
\\
W^{\text{per}}_N(i,j) &=&
\begin{cases}
        X_N(i,j) &, \text{~  if ~}  P_N(i,j) = 1\\
         0 &, \text{~  else ~}
         \end{cases} \,.
\end{eqnarray*}
We call these ensembles centred if the mean $m$ vanishes.
\item[c)] We call $b_N$ the {\em half-width} and $w_N:= \min(N, 2 b_N +1)$ the {\em (maximal)
bandwidth} of the band matrices defined above.
\end{itemize}
\end{definition}
\begin{remark}
\label{remark:wN}
Observe that for periodic band matrices every row (and every column) of $P_N$ has the same
number $w_N$ of non-zero entries.
Therefore the case of bandwidth $w_N=N$ makes $W^{\text{per}}_N$ a full Wigner matrix.

For strict band matrices, however, these ensembles may differ from each other
even if $w_N = N$, depending  on the value of the half-width $b_N \geq (N-1)/2$.
In this situation $w_N=N$ is the maximal number of
non-zero entries that a row (column) of $B_N$ may have.
Therefore $w_N$ was named the {\em maximal bandwidth}. The band matrix $W_N$
 is a full Wigner matrix only in the case $b_N=N-1$.
\end{remark}

We end this subsection by noting a few spectral properties
of the band matrices $B_N$ and $P_N$ for later reference.

\begin{lemma}\label{lemma:G1}
Let $B_N$ and $P_N$ be defined as in Definition \ref{def:eeebm}a)
with bandwidth $w_N$ and half-width $0 \leq  b_N \leq N-1$.

a) For $0 \leq j \leq N-1$ define $\omega_j := j\pi/N$ and
$u^{(j)} \in \mathbb{C}^N$ by $u^{(j)}_k := e^{2 i \omega_j k}$,
$1\leq k \leq N$.
%\begin{equation*}
%\omega_j := j\pi/N \quad \text{and} \quad
%u^{(j)} \in \mathbb{C}^N \text{ ~ by ~} u^{(j)}_k := e^{2 i \omega_j k}\,, 1\leq k \leq N \,.
%\end{equation*}
The vectors $(u^{(j)})_j$ form an orthogonal basis of eigenvectors of $P_N$
and the corresponding eigenvalues $\mu_j$ are given by $\mu_0 = w_N$
and $\mu_j =\frac{\sin (\omega_j w_N)}{\sin \omega_j}$ for $j \geq 1$.

b) The $\ell_2$-operator norms satisfy
\begin{equation*}
\| P_N \|_{op} = w_N \qquad \text{and} \qquad
w_N(1-\delta_N) \leq \| B_N \|_{op} \leq  w_N \quad \text{with}
\end{equation*}
$\delta_N=\frac{w_N}{4 N}$ \, if \, $2 b_N+1 \leq N$ and $\delta_N=\left(1-\frac{b_N}{N}\right)^2$ else.

\end{lemma}
\begin{proof}
Statement a) can be verified by computation. In order to see claim b) recall first that the
modulus of any eigenvalue of a given matrix $(A(i,j))_{i, j}$
is bounded above by max$_i \sum_j |A(i,j)|$ (e.g.~consider the eigenvalue equation
for a component for which the eigenvector
has maximal modulus). Thus both the operator norms of $B_N$ and $P_N$
are bounded above by $w_N$.
Secondly, for real symmetric matrices $A$ the operator norm is bounded below by
$(v, Av)/(v,v)$ for any non-zero vector $v$.
Choose $v = (1, \ldots, 1)$. Then $(v, Av)$ is just the number of non-zero entries for
$A \in \{P_N, B_N\}$. In the case of $P_N$ this number is $w_N N$.
For $B_N$ this number is $w_N N -k_N(k_N+1)$ with $k_N = \min(b_N, N-b_N-1)$.
Using $k_N(k_N+1)\leq(k_N+\frac{1}{2})^2=w_N^2/4$ for $2 b_N+1 \leq N$ and
$k_N(k_N+1)\leq(k_N+1)^2=(N-b_N)^2$ for $2 b_N+1 > N$
completes the proof.
\end{proof}

\subsection{Semicircle for band matrices}

For the ensembles that we have defined in the previous subsection we now formulate our main
result on the limiting spectral distribution. Our proof starts from the special case of centred
ensembles where the result is known. The extension to arbitrary means uses
Proposition \ref{prop:app2.1} which provides estimates on the effects on the spectral measure
of adding matrices of small operator norm or of small rank.

\begin{theorem}\label{thm:wignerband}
  Let $0 \leq b_N \leq N-1$ be a given sequence with $b_N \to \infty$ for $N \to \infty$. Recall the
  notion of Wigner band matrices with half-width $b_N$ and bandwidth $w_N= \min(N, 2 b_N +1)$
  from Definition \ref{def:eeebm}.\vspace{3pt}\\
  a) We distinguish the two cases of periodic and strict band matrices.
    \begin{enumerate}
       \item Assume that the entries of the \emph{periodic} Wigner band matrices $W^{\text{per}}_N$
       have variance $v$ and arbitrary mean $m$. Then the empirical eigenvalue measures $\mu^{\omega}_{N} $
        of $\frac{1}{\sqrt{w_{N}}}W^{\text{per}}_N$ converge weakly in probability
        to the semicircle law $\sigma_{v}$.
        \item Statement~1 also holds for the empirical eigenvalue measures $\mu^{\omega}_{N}$ of the \emph{strict} band matrices  $W_N/\sqrt{w_N}$
   if we require in addition that the scaled half-widths $\frac{b_N}{N}$ converge either to $0$ or to $1$ for $N \to \infty$.
            \end{enumerate}
 b) Let us add to the general assumption $b_N \to \infty$ above the summability condition  $\sum_{N} (N b_N)^{-1} < \infty$.
     Then both statements of part a) remain true if we strengthen the assertion of weak convergence in probability to
     weak convergence $\IP$-almost surely.
 \end{theorem}

Before we set out to prove the theorem
for arbitrary values of the mean $m$, let us briefly describe what is known in the case of centred entries. In this case
statement a) of this theorem is due to \cite{BMP} and \cite{MPK}.
For periodic band matrices with centred entries statement b) has been observed in \cite{Fleermann}
as a special case of matrix ensembles with almost uncorrelated entries,
 see also \cite[Problem 2.4.13]{Pastur Sherbina}.
 For the convenience of the reader
 we sketch a proof of part b) for the centred case $m=0$ in Section \ref{app:sec1} of the Appendix.
Finally, we mention that for strict Wigner band matrices
with centred entries \cite{BMP, Catalano, Pastur Sherbina}
also treat the case where
$\lim_{N \to \infty} \frac{b_N}{N}$ exists and where the limit lies in the open interval $(0, 1)$.
In this situation the empirical eigenvalue measures of
$\frac{1}{\sqrt{w_{N}}}W_N$ still converge but not to a semicircle law.

In order to derive Theorem \ref{thm:wignerband} from its specialized version with centred entries
we proceed as in the proof of Proposition \ref{prop:convrho} for full matrices.
We split off a matrix $M_{N}$ containing the expectations
of the matrix elements. In the case of full random matrices $M_{N}$ turned out to be a matrix of rank
one, in fact $M_{N}\,=\,m\,\cE_{N}$ (see \eqref{eq:eN}) with $m$ being the mean of the entries.
In the case of band matrices we obtain instead $M_{N}\,=\,m\, P_{N}$ or $M_{N}\,=\,m\, B_{N}$
for periodic or strict band matrices respectively.

The simple `rank-one'-argument of Proposition \ref{prop:convrho} cannot work in the
case of band matrices, since in this case the corresponding matrices $M_{N}$ do not
have bounded rank, they may even have full rank $N$. However, we will develop a
more refined argument that is based on Lemma~\ref{lemma:G1} and on the following observation.

\begin{proposition}\label{prop:app2.1}
Let $A$, $R$ be real symmetric $N \times N$ matrices and denote by $\rho$ and $\mu$
the eigenvalue distribution measures of $A$ and $B:= A + R$ respectively. Then for every bounded function $f \in C^1(\mathbb R)$ the following estimates hold.
\begin{itemize}
\item[a)]$\quad
\left| \int f d\rho - \int f d\mu \right| \leq  \|R\|_{op} \|f'\|_{L_{\infty}}$.
\item[b)]$\quad
\left| \int f d\rho - \int f d\mu \right| \leq 2 \frac{\textrm{ \emph{rank}}(R)}{N} \| f'\|_{L_1}$.
\end{itemize}
\end{proposition}
\begin{proof}
Denote by $\rho_1 \leq \cdots \leq \rho_N$ and $\mu_1 \leq \cdots \leq \mu_N$ the eigenvalues of $A$ and $B$ respectively.
It is well known that the minmax principle allows to compare the spectra of the matrices $A$ and $B$ in terms of the
operator norm of $R=B-A$, leading to statement a), and in terms of the rank of $R$ which is the basis for statement b).

a) Since $|\rho_j-\mu_j| \leq \|R\|_{op}$ for all $j$ we have $|f(\rho_j)-f(\mu_j)| \leq \|R\|_{op}\|f'\|_{L_{\infty}}$
and the claim follows by summation over $j$.

b) Denote $r:=$ rank$(R)$. Then for all $1 \leq j \leq N$ the eigenvalue $\mu_j$ lies in the interval $[\rho_{j-r}, \rho_{j+r}]$
where we set $\rho_i := - \infty$ for $i \leq 0$ and $\rho_i := \infty$ for $i \geq N+1$. Hence
\begin{equation*}
|f(\rho_j)-f(\mu_j)| \leq \left| \int_{\mu_{j}}^{\rho_{j}} |f'(x)| dx \right| \leq \int_{\rho_{j-r}}^{\rho_{j+r}} |f'(x)| dx
\end{equation*}
and summation over $j$ yields statement b) via
\begin{equation*}
\sum_{j=1}^N |f(\rho_j)-f(\mu_j)| \leq \sum_{j=1}^N \sum_{k=-r}^{r-1} \int_{\rho_{j+k}}^{\rho_{j+k+1}} |f'| \leq
2r \sum_{l=0}^{N} \int_{\rho_{l}}^{\rho_{l+1}} |f'| = 2r \| f'\|_{L_1}\,.
\end{equation*}
\end{proof}

In order to prove Theorem \ref{thm:wignerband} we provide a second auxiliary result that shows how the estimates of Proposition
\ref{prop:app2.1} can be used to conclude persistence of weak convergence in probability.

\begin{lemma}\label{lemma:persistence}
Let $(\Omega, \mathcal{F}, \IP)$ be a probability space and let $\mu_N^{\omega}$, $\rho_N^{\omega}$, $\mu$
be probability measures on $(\IR, B(\IR))$ for every $\omega \in \Omega$. Assume that $(\rho_N^{\omega})_N$
converges weakly in probability to $\mu$. Moreover, suppose that there exists a real-valued sequence $(c_N)_N$
that converges to $0$ such that for all $\omega \in \Omega$ and all functions $f \in \mathcal{S}$ with
$$
\mathcal{S} := \{ f \in C^1(\mathbb R) \colon    \| f'\|_{L_1} + \|f'\|_{L_{\infty}} < \infty \}
$$
we have
\begin{equation}\label{eq:app2.0}
\left| \int f d\rho_N^{\omega} - \int f d\mu_N^{\omega} \right| \leq c_N \left(  \| f'\|_{L_1} + \|f'\|_{L_{\infty}} \right)\,.
\end{equation}
Then $(\mu_N^{\omega})_N$ also
converges weakly in probability to $\mu$.
\end{lemma}
\begin{proof}
Fix $f \in C_b (\IR)$ and $\epsilon > 0$. Since the conditions of Definition \ref{def:Tprobwconv} for weak convergence in probability
are trivially satisfied in the case $f=0$ one may assume $\|f\|_{L_{\infty}} > 0$. We first approximate $f$ by differentiable functions on suitable compact sets
that depend on the given value of~$\epsilon$.

Since $\mu$ is a probability measure a number $R >0$ can be picked such that
\begin{equation}\label{eq:app2.1}
\mu(\mathbb R \setminus [-R, R]) \leq \frac{\epsilon}{8 \|f\|_{L_{\infty}}}\,.
\end{equation}
Then choose $g \in C^1(\mathbb R)$ with
\begin{equation}\label{eq:app2.2}
\sup \{ |f(x) - g(x)| \colon | x | \leq R+1 \} \leq \frac{\epsilon}{8}\,,
\end{equation}
and a smooth cut-off function $\chi \colon \mathbb R \to [0, 1]$ that satisfies
\begin{equation}\label{eq:app2.3}
\chi(x) =
\begin{cases}
        1 &, \text{~  if ~}  | x | \leq R \,,\\
        0 &, \text{~  if ~}  | x | > R+1 \,.
         \end{cases}
%\qquad \text{and} \qquad
%\| \chi '\|_{L_1} + \|\chi '\|_{L_{\infty}} \leq 4
%\,.
\end{equation}
Write $\int f d\mu_N^{\omega} - \int f d\mu = \sum_{i=1}^{4} \Delta_i^{\omega}$ with
\begin{align*}
\Delta_1^{\omega}:= \int f (1- \chi) d\mu_N^{\omega} - \int f (1- \chi) d\mu
\,,\quad &
\Delta_2^{\omega}:= \int g \chi d\mu_N^{\omega} - \int g \chi d\rho_N^{\omega}
\,,\\
\Delta_3^{\omega}:= \int (f-g) \chi d\mu_N^{\omega} - \int (f-g) \chi d\rho_N^{\omega}
\,, \quad &
\Delta_4^{\omega}:= \int f \chi d\rho_N^{\omega} - \int f \chi d\mu
\,.
\end{align*}
We estimate
\vspace{3pt}\\
$|\Delta_1^{\omega}| \leq \|f\|_{L_{\infty}} (|\int (1- \chi) d\mu_N^{\omega} |+|\int (1- \chi) d\mu |)
\leq \|f\|_{L_{\infty}}(\Gamma_1^{\omega} +\Gamma_2^{\omega} + 2\Gamma_3)$\vspace{13pt}\\
where  \vspace{-23pt}
\begin{eqnarray*}
\Gamma_1^{\omega} &:=& \Big|\int (1- \chi) d\mu_N^{\omega} - \int (1- \chi) d\rho_N^{\omega} \,\Big|\,,\\
\Gamma_2^{\omega} &:=& \Big|\int (1- \chi) d\rho_N^{\omega} - \int (1- \chi) d\mu \,\Big|\,,\\
\Gamma_3 &:=& \Big|\int (1- \chi) d\mu\,\Big|\,.
\end{eqnarray*}
Use \eqref{eq:app2.1}, \eqref{eq:app2.3} to bound $\Gamma_3$, \eqref{eq:app2.2}, \eqref{eq:app2.3} for $|\Delta_3^{\omega}|$,
and \eqref{eq:app2.0} for $|\Delta_2^{\omega}|$ and $\Gamma_1^{\omega}$.

By the hypothesis of Lemma \ref{lemma:persistence} there exists $N_0$
such that for all $N \geq N_0$
\begin{equation*}
c_N \left[
\| (g\chi)'\|_{L_1} + \|(g\chi)'\|_{L_{\infty}} + \|f\|_{L_{\infty}} \left(
\| (1-\chi)'\|_{L_1} + \|(1-\chi)'\|_{L_{\infty}}
\right)
\right] \leq \frac{\epsilon}{4}\,.
\end{equation*}
Combining all these estimates we obtain for all $\omega \in \Omega$ and all $N \geq N_0$:
\begin{equation*}
\left| \int f d\mu_N^{\omega} - \int f d\mu \right| \leq \frac{3}{4}\epsilon +
\left| \Delta_4^{\omega} \right| +
\|f\|_{L_{\infty}}  \Gamma_2^{\omega}
\,.
\end{equation*}
Since $(\rho_N^{\omega})_N$ converges weakly in probability to $\mu$ one easily concludes from the last inequality that
$\IP (|\int f d\mu_N^{\omega} - \int f d\mu| > \epsilon) \to 0$ as $N \to \infty$.
\end{proof}

Now we have gathered all the technical ingredients to derive the statement of
Theorem \ref{thm:wignerband} from its special version with centred entries.\\
\textbf{Proof of Theorem \ref{thm:wignerband}.}
We begin with the periodic case.
Let $\Pi_{N}$ be the orthogonal projection on the spectral subspace
of $P_{N}$ with respect to the eigenvalues with absolute value $\leq \sqrt[4]{w_{N}}$ and set
$S_{N}:=P_{N}\Pi_{N}$ and $R_{N}:=P_{N}\big(1-\Pi_{N} \big)$.
Then $\|S_{N}\|_{op}\leq \sqrt[4]{w_{N}}$.
Moreover, the rank of $R_{N}$ equals the number $r$ of eigenvalues
of $P_{N}$ larger than $\sqrt[4]{w_{N}}$. By Lemma \ref{lemma:G1} the number $r$ can be estimated:
\begin{align}
   r~\leq~\Big|\Big\{j\in\{0,1,\ldots,N-1\}\;\big|\;\big. \sin(\frac{j\pi}{N})\;<\;
   \frac{1}{\sqrt[4]{w_{N}}} \Big\}\Big|
\end{align}
Using $\sin(\pi x) \geq 2x$
 for $x \in [0, 1/2]$ one may deduce that $r\leq 1+ N/\sqrt[4]{w_N}$.\vspace{3pt}

Denote by $\mu_{N}^{\omega}$ the eigenvalue distribution measure
of $W^{\text{per}}_{N}/\sqrt{w_{N}}$
and by $\rho_{N}^{\omega}$ the eigenvalue distribution measure
of $(W^{\text{per}}_{N}-mP_{N})/\sqrt{w_{N}}$ where again~$m$
denotes the mean of the matrix entries.
%Let $f\in C_{0}^{\infty}$ be an infinitely differentiable function with compact support.
Applying Proposition~\ref{prop:app2.1} twice we obtain for all bounded functions $f\in C^{1}(\mathbb{R})$
and for all $\omega \in \Omega$:
\begin{align}
   &\big|\int f\, d\rho_{N}^{\omega}~-~\int f\, d\mu_{N}^{\omega}\big| \nonumber\\
~\leq~&\frac{m}{\sqrt{w_{N}}}\|S_{N}\|_{op}\; \|f'\|_{L_\infty}~+~\frac{2}{N}\text{ rank}\,(R_{N})\;\|f'\|_{L_1} \label{eq:integral_estimate}\\
~\leq~&\frac{m}{\sqrt[4]{w_{N}}}\; \|f'\|_{L_\infty}~+~2\big(\frac{1}{\sqrt[4]{w_{N}}}+
\frac{1}{N}\big)\;\|f'\|_{L_1} \nonumber
%\\
%\to &\quad 0\,.
\end{align}
Since $W^{\text{per}}_{N}-mP_{N}$ corresponds to the centred case for which we know Theorem \ref{thm:wignerband}
to hold, we have the desired convergence of $\rho_N^{\omega}$ to the semicircle $\sigma_v$. In order to transfer this
result to the eigenvalue distribution measures $\mu_{N}^{\omega}$ of $W^{\text{per}}_{N}/\sqrt{w_{N}}$ we distinguish
between statements a) and b) of the Theorem.

For a) the claim follows from Lemma \ref{lemma:persistence}
and estimates \eqref{eq:integral_estimate}. In the situation of b) we proceed differently. Let $\omega \in \Omega_0$ be
contained in the set of full measure for which $(\rho_N^{\omega})_N$ converges weakly to the semicircle law. Using in addition
inequalities \eqref{eq:integral_estimate} we deduce for all infinitely differentiable functions with compact support
$f \in C_{0}^{\infty}(\mathbb{R})$ that
$$\int f\,d\mu_{N}^{\omega}~\to~\int f\,d\sigma_{v}$$
as $N \to \infty$. Hence vague convergence of $(\mu_N^{\omega})_N$ is established. As the limiting measure $\sigma_v$
 is a probability measure, vague convergence implies weak convergence
$\mu_N^{\omega} \Rightarrow \sigma_v$ for all  $\omega \in \Omega_0$.

Finally, we turn to the case of strict band matrices. Observe that
\begin{align}
   \text{rank}\,\big(P_{N}-B_{N}\big)~\leq~2 \min (b_N, N -b_N - 1)\,.
\end{align}
If $\frac{b_{N}}{N}$ converges to either $0$ or $1$ we conclude that
$\frac{1}{N}$rank$(W^{\text{per}}_{N}-W_N) \to 0$ as $N \to \infty$.
Thus statement~2 of Theorem~\ref{thm:wignerband} can be deduced from statement~1
via Proposition~\ref{prop:app2.1}b) in the same way as statement~1 was inferred
from the case of centred ensembles above.
\hfill \rule{0.5em}{0.5em}

\section{Exchangeable Random variables and de Finetti matrix ensembles}
\label{sec:DFE}

The main result in this section is Theorem \ref{thm:DFB-semicircle} that shows for
large classes of band matrices with exchangeable entries that include in particular the case
of full matrices that the empirical eigenvalue measures of the appropriately rescaled
random matrices still converge to a semicircle $\sigma_V$. In contrast to the
Wigner case the scale $V$ of the semicircle is now random.

In Subsection \ref{ssec:3.3} we obtain a deterministic
limit law by downgrading the quality of convergence from almost sure convergence to
convergence in expectation. In addition, we can characterize all cases
for which the deterministic law is a semicircle.

In order to get started we first explain de Finetti's Theorem in a somewhat generalized form
that links sequences of exchangeable real-valued random variables to i.i.d.~sequences.

\subsection{Exchangeable random variables}\label{ssec:3.1}

After the definition of exchangeable sequences of random variables we discuss de Finetti's
theorem on their representation as averages of i.i.d.~sequences in a form that
is suitable for the present paper. As a first application we prove a strong law of  large numbers
that differs from the classic result for i.i.d.~variables only in the fact that the limit may be a random
variable rather than a constant.

\begin{definition}\label{def:exchangeable}
A finite sequence $(\xi_i)_{1\leq i\leq N}$ of random variables with underlying probability space $(\Omega, \mathcal{F}, \IP)$ is called {\em exchangeable},
if for all permutations $\pi$ on $\{1, \ldots, N\}$, and all $F \in \mathcal{F}$ it is true that
\begin{equation*}
\IP \big( (\xi_1, \ldots, \xi_N) \in F \big) =  \IP \big( (\xi_{\pi(1)}, \ldots, \xi_{\pi(N)}) \in F \big) \,.
\end{equation*}

An infinite sequence $(\xi_i)_{i\in\IN}$ is called
{\em exchangeable} if the finite sequences $(\xi_i)_{1\leq i\leq N}$
are exchangeable for all $N$.
\end{definition}

A celebrated result of de Finetti \cite{deFinetti, deFinetti2}
characterizes infinite exchangeable sequences with values in
$\{-1,+1\} $.
For each such sequence $\{\xi_{i}\}$ there is a probability
measure $\mu $ on
$[-1,1]$ such that
\begin{align}
   \IP\Big(\{\xi_{i}\}\in F\Big)~=~\int P_{t}\left(F\right)\,d\mu(t)
\end{align}
where $P_{t}$ is the infinite product $\bigotimes_{i\in\IN} \lambda_{t}$ on $\{-1,+1\}^{\IN} $ of
the measures $\lambda_{t}$ on $\{-1,+1\} $ given by
$\lambda_{t}(\{1\})=\frac{1}{2}(1+t)$ and
$\lambda_{t}(\{-1\})=\frac{1}{2}(1-t)$.

Hewitt-Savage \cite[Theorem 7.4]{HewittSavage} extended de Finetti's theorem to exchangeable
sequences with values in rather general spaces. We will need
here only the case of $\IR$-valued random variables and formulate it in a form which we found
convenient for our purpose.

For any probability measure $\lambda$ on $\IR$ (as always equipped with the Borel $\sigma$-algebra)
we denote by $P_{\lambda}$ the product measure $\bigotimes_{i\in\IN}\lambda$ on $\IR^{\IN}$.

We also denote by $\mathcal{M}_{1}=\mathcal{M}_{1}(\IR)$ the space of all probability measures
on $\IR$ equipped with the topology of weak convergence.

\begin{theorem}\label{th:Finetti}
   Let $(\xi_{i})_{i\in\IN}$ be an exchangeable sequence of $\IR$-valued random variables on
   the probability space $(\Omega,\cF,\IP)$.

   Then there is a probability space $(T,\cT,\mu)$ and a measurable mapping $\Lambda:T\to\cM_{1}(\IR)$
    such that
      \begin{align}\label{eq:Finetti}
      \IP\big(\{\xi_{i}\}\in F\big)~=~\int P_{\Lambda_{\tau}}(F)\;d\mu(\tau)
   \end{align}
   We call the probability measure $\mu$ the \emph{de Finetti measure} associated with
   the sequence $(\xi_{i})_i $.
\end{theorem}
For details see e.\,g. \cite{Aldous}.
\begin{remark}\label{rem:spin}
\begin{enumerate}
   \item For $T$ one can always choose $T=\cM_{1}$ and for $\Lambda$ the identity.
   Usually
   the generalized de Finetti Theorem is formulated with this choice. For our purpose
   we prefer the above equivalent but somewhat more flexible version.
   \item When the $\xi_{i} $ have values
   in $\{-1,1\}$ we recover de Finetti's original case and we may chose\,
   $T=[-1,1]\cong \cM_{1}(\{-1,1\})$. We refer to this as the \emph{spin case}.
   \item The members of an exchangeable sequence are identically distributed but in general not independent.
   \item Observe that $\IE\big(|\xi_1|^p\big) < \infty$ for some $0 < p < \infty$ implies that the $p$-th
   moment of $\Lambda_{\tau}$ exists for $\mu$-almost all $\tau$. This will be used in the following convention.
\end{enumerate}

\end{remark}

\begin{convention}\label{not:fm}
When speaking of an exchangeable sequence of real-valued random variables $(\xi_{i})_i$ in the following
we will always tacitly suppose that all moments of $\xi_{1}$
(and therefore of all $\xi_{i}$ by Remark \ref{rem:spin}.3) are finite.
Due to the last observation in Remark \ref{rem:spin} we may and will assume
that the moments of the corresponding measures $\Lambda_{\tau}$
(as in \eqref{eq:Finetti}) are finite for \emph{all} $\tau\in T$.
\end{convention}

For $\tau \in T$ we introduce  the  moments and the variance of $\Lambda_{\tau}$:
\begin{align}\label{eq:Tex.3}
m_k (\tau) &:= \int x^k ~ d \Lambda_{\tau} (x)\,,
%\mbox{ for } k \in \{1, 2\}  \,,
\\
\label{eq:Tex.4}
v (\tau) &:=  m_2 (\tau) - m_1 (\tau)^2 \,.
\end{align}
According to Convention \ref{not:fm} the moments $m_{k}(\tau)$ are finite for all $\tau \in T$ and all $k$.

Furthermore, denote by $\mu_1$ resp. $\nu$ the push forwards of the measure $\mu$ on $T$ under the maps
$\tau \mapsto m_1 (\tau)$ resp. $\tau \mapsto v (\tau)$.
%\marginpar{Vielleicht sind die Pushforwards nicht n\"{o}tig?}
Observe that $\mu_1, \nu$ are both probability measures on $\IR$ with $\supp (\nu) \subset [0, \infty)$.
For the {\em spin case} defined in Remark \ref{rem:spin} it is straightforward to compute $m_1 ( \tau) = \tau$,
thus $\mu_1 = \mu$, and $v(\tau) = 1 - \tau^2$.

The following proposition formulates a strong law of large numbers for sequences of exchangeable
$\mathbb{R}$-valued random variables. As we see below it is a simple consequence of the corresponding
classic law for
i.i.d.~sequences. Nevertheless, there is a significant difference between these two cases. For exchangeable
sequences the limit is generally not a number but a random variable.
\begin{proposition}\label{prop:exchangeable_mean}
Let $( \xi_i )_{i }$ be a sequence of exchangeable $\mathbb{R}$-valued random variables and recall
\eqref{eq:Finetti} as well as Convention \ref{not:fm}.
%Furthermore, let $(L_N)_N$ be a sequence of positive integers with $L_N \to \infty$ for $N \to \infty$.
Define for $n \in \mathbb{N}$ the random variables
\begin{eqnarray*}
M_n &:=& \frac{1}{n} \sum_{i=1}^{n} \xi_i \,.
%V_n &:=& \frac{1}{n} \sum_{i=1}^{n} \xi_i^2 - \Big(
%\frac{1}{n} \sum_{i=1}^{n} \xi_i
%\Big)^2\,.
\end{eqnarray*}
Then the sequence $(M_n)_n$ converges $\mathbb{P}$-almost surely to a random variable $M$.
Moreover, the limit satisfies $M=m_1(\tau)$
almost surely with respect to $P_{\Lambda_{\tau}}$. Thus the law for the random variable
$M$ is given by the push forward $\mu_1$ of the first moment $m_1$.
\end{proposition}
\begin{proof}
Under the probability measure $P_{\Lambda_{\tau}}$ the random variables
   ${M_{n}}$ converge to ${m_1(\tau)}$ almost surely by the classic strong law of large numbers. Since
   \begin{align*}
      \IP\big({M_{N}}\to M\big)~=~\int P_{\Lambda_{\tau}}\big({M_{N}}\to M\big)\,d\mu(\tau) \,
   \end{align*}
   the claim follows.
   \end{proof}

Applying Proposition \ref{prop:exchangeable_mean} in addition to the squares of the
random variables, which also form an exchangeable sequence, we obtain
\begin{proposition}\label{prop:exchangeable_variance}
In the situation of Proposition \ref{prop:exchangeable_mean} define for $n \in \mathbb{N}$ the random variables
\begin{eqnarray*}
V_n &:=& \frac{1}{n} \sum_{i=1}^{n} \xi_i^2 - \Big(
\frac{1}{n} \sum_{i=1}^{n} \xi_i
\Big)^2\,.
\end{eqnarray*}
Then the sequence $(V_n)_n$ converges $\mathbb{P}$-almost surely to a random variable $V$.
Moreover, the limit satisfies $V = v(\tau)$
almost surely with respect to $P_{\Lambda_{\tau}}$ and the law for the random variable
$V$ is given by the push forward $\nu$ of the variance $v$.
\end{proposition}
\begin{definition}\label{def:empirical}
For sequences $( \xi_i )_{i }$ of exchangeable $\mathbb{R}$-valued random variables we call the
random variable $M$ defined in Proposition \ref{prop:exchangeable_mean} the
{\em (limiting) empirical mean} and the random variable $V$ defined in Proposition \ref{prop:exchangeable_variance} the
{\em (limiting) empirical variance}.
\end{definition}

\subsection{De Finetti band ensembles and the random semicircle law}

In this subsection we transfer the assertion of Theorem \ref{thm:wignerband} to band matrices
with entries drawn from an exchangeable sequence that we call de Finetti band ensembles.

\begin{definition}\label{def:dfbm}
Let $(b_N)_N$ be a sequence of integers with $0 \leq  b_N \leq N-1$
and denote the band matrices $B_N$ and $P_N$ as in Definition \ref{def:eeebm}a).
An ensemble of a familiy of $N \times N$ real symmetric matrices
$(F_N)_N$ / $(F^{\text{per}}_N)_N$  is called a {\em strict de Finetti band ensemble} / a
{\em periodic de Finetti band ensemble} if the entries of $F_N$ / $F^{\text{per}}_N$ are zero
whenever the corresponding  entries of $B_N$ / $P_N$ are zero and if all other entries of
$F_N(i,j)$ / $F^{\text{per}}_N(i,j)$ are filled for each $N$, $1 \leq i \leq j \leq N$, by the first entries of a fixed
sequence of exchangeable $\mathbb{R}$-valued random variables.
The remaining entries are then determined by the symmetry of the matrix.
\end{definition}
\begin{remark}
We recall that we always assume that Convention \ref{not:fm} is satisfied.
\end{remark}
\begin{remark}\label{rem:LN}
Note that the exchangeability of the random variables $(\xi_i)_i$ implies that the ensemble depends neither on
the set of $\xi_i$ that is selected to fill the $N$-th matrix (as long as different $\xi_i$s are used for different entries of the
upper triangular parts of the random matrices),
nor on the specific order in which we fill the matrix. For example, we could fill it row-wise or sub-diagonal by sub-diagonal.
%Since the random
%variable are exchangeable the result is independent of the specific algorithm of filling the matrix.
For other sequences of random variables the way of filling the matrix may be crucial (see \cite{LoeweS}).

%Let us denote by $L^{\text{per}}_N$ and $L_N$ the numbers of different members from the sequence of exchangeable random variables
%that are needed to build $F^{\text{per}}_N$ and $F_N$ respectively. It is given by $L^{\text{per}}_N=N(w_N+1)/2$
%and $L_N= L^{\text{per}}_N -k_N(k_N+1)/2$ respectively, where $k_N =$ min$(b_N, N-b_N-1)$.
\end{remark}

We are now ready to state our main result on the limiting spectral density of de Finetti band matrices. Observe that this result
also includes the case of full de Finetti ensembles by choosing the bandwidths sufficiently large.

\begin{theorem}\label{thm:DFB-semicircle}
Let $(b_N)_N$ be a sequence of integers with $0 \leq  b_N \leq N-1$ and $b_N \to \infty$.
Recall from Definition \ref{def:dfbm} the meaning of the corresponding de Finetti band ensembles
$(F_N)_N$ and $(F^{\text{per}}_N)_N$ with bandwidth $w_N:= \min(N, 2 b_N +1)$ (Definition \ref{def:eeebm}).
Denote by $(\xi_i)_i$ the sequence of exchangeable random variables from which
the entries of the matrices are drawn and let $V$ be its empirical variance (Definition \ref{def:empirical}).\\
a) We distinguish the two cases of periodic and strict band matrices.
\begin{enumerate}
   \item The empirical eigenvalue measures $\mu^{\omega}_{N}$   of the periodic band matrices $F^{\text{per}}_N/\sqrt{w_N}$
   converge weakly in probability to the (random semicircle)
   measure $\sigma_{V(\omega)}$.
   \item Statement~1 also holds for the empirical eigenvalue measures $\mu^{\omega}_{N}$ of the strict band matrices  $F_N/\sqrt{w_N}$
   if we require in addition that the scaled half-widths $\frac{b_N}{N}$ converge either to $0$ or to $1$ for $N \to \infty$.
   \end{enumerate}
b) Let us add to the general assumption $b_N \to \infty$ above the summability condition  $\sum_{N} (N b_N)^{-1} < \infty$.
     Then both statements of part a) remain true if we replace the assertion of weak convergence in probability by
     weak convergence $\IP$-almost surely.
\end{theorem}
\begin{proof}
   By assumption the measure $\IP$ associated with either the periodic or the strict de Finetti band ensembles
   have the form \eqref{eq:Finetti}:
     \begin{align}
   \IP\Big((X_{i_{1}j_{1}}, \ldots ,X_{i_{k}j_{k}})\in A\Big)~=~\int P_{\Lambda_{\tau}}(A)\;d\mu(\tau)\,.
\end{align}
for $\{i_{1},j_{1}\},\ldots,\{i_{k},j_{k}\}$ pairwise distinct and for every Borel set $A\subset\IR^k$.
Fix $\tau \in T$ and consider the ensembles $(F_N)_N$ and $(F^{\text{per}}_N)_N$ with respect to the probability measure
$P_{\Lambda_{\tau}}$. Then they are strict and periodic Wigner band matrices respectively with variance $v(\tau)$
(see Definition \ref{def:eeebm}). We first consider part b). Then
Theorem \ref{thm:wignerband} b) implies that the empirical eigenvalue measures $\mu^{\omega}_{N}$
converge in both cases~1 and ~2 to $\sigma_{v(\tau)}$ almost surely
with respect to $P_{\Lambda_{\tau}}$. By Proposition \ref{prop:exchangeable_variance} we have in addition
$P_{\Lambda_{\tau}}$-almost surely that $v(\tau)=V(\omega)$. Hence
\begin{align}
   \IP\big(\mu^{\omega}_{N}\Rightarrow \sigma_{V(\omega)}\big)~=~
   \int P_{\Lambda_{\tau}}\big(\mu^{\omega}_{N}\Rightarrow \sigma_{V(\omega)}\big)\;d\mu(\tau)~=~1\,.
\end{align}

In order to prove part a) fix  $f \in C_b (\IR)$ and $\epsilon > 0$. Then part a) of Theorem \ref{thm:wignerband} implies
for all $\tau \in T$ that
\begin{equation*}
P_{\Lambda_{\tau}} \Big(\;\Big| \int f (x) ~ d\mu_N^\omega (x) - \int f (x) ~ d \sigma_{v(\tau)} (x)\, \Big| > \epsilon \Big) \rightarrow 0
\end{equation*}
as $N\to\infty$. Integration of this relation over $T$ with respect to the de~Finetti measure $\mu$ together with Lebesgue's theorem
of dominated convergence yield the claim.
\end{proof}

The proof of Theorem~\ref{thm:DFB-semicircle} also implies the following conditional convergence to deterministic semicircle laws.
\begin{corollary}
Recall the definition of the push forward measure $\nu$ below equation \eqref{eq:Tex.4}.
Under the assumptions and with the notation of Theorem~\ref{thm:DFB-semicircle} b) we have for $\nu$-almost all $v \in [0, \infty)$:
    \begin{align*}
      \IP\,\Big( \mu^{\omega}_{N} \Rightarrow\sigma_{v}\,\Big|\,V(\omega)=v\Big.\Big)
      ~=~1\,.
   \end{align*}
\end{corollary}
%\begin{proof}
%You must ascend by foot on a special day in September to a special corner of a special bar on the top of Bergamo mountain.
%\marginpar{is there a better way to say this? :-)}
%\end{proof}

\subsection{Expected limiting spectral density}\label{ssec:3.3}

The weak almost sure convergence as well as the weak convergence in probability asserted in Theorem~\ref{thm:DFB-semicircle}
both imply weak convergence in expectation
(see Definition \ref{def:Tprobwconv}). Applying Fubini's theorem to the right hand side of \eqref{eq:convwe}
one may choose the limiting measure to be deterministic by taking the expectation of the limiting measures $\sigma_{V(\omega)}$.
The result of this averaging is the
measure $\sigma_{\mu}$ on $\IR$ that we define via the Riesz representation theorem through
\begin{align}
\int f(x)\; d\sigma_{\mu}(x)~:&=~\IE\,\Big(\int f(x)\;d\sigma_{V(\omega)}(x)\Big)\\ \label{eq:B2.1}
&=~ \int_{T} \int f(x)\; d\sigma_{{v(\tau)}}(x)\, d\mu (\tau)
\end{align}
for each bounded continuous function $f\in C_{b}(\IR)$.
The last equality follows from Proposition \ref{prop:exchangeable_variance}. We summarize:

\begin{theorem}\label{thm:Expconv}
   Under the assumptions and with the notation of Theorem~\ref{thm:DFB-semicircle}
   the empirical eigenvalue measures $\mu^{\omega}_{N} $ converge weakly in expectation to
   the measure $\sigma_{\mu}$.
\end{theorem}

Next we derive a representation for the limiting measure $\sigma_{\mu}$.

\begin{proposition}
\label{prop:B1}
The Borel probability measure $\sigma_{\mu}$ that Equation \eqref{eq:B2.1} defines on $\IR$
is symmetric in the sense $\sigma_{\mu}(A) = \sigma_{\mu}(-A) $. Moreover,
\begin{align*}
\sigma_{\mu} = \mu (\{v(\tau)=0\})\, \delta_0 + \sigma_{\mu}^{abs}
\end{align*}
with $\sigma_{\mu}^{abs}$ being absolutely continuous with Lebesgue density
\begin{equation}\label{eq:Tex.5}
\rho_{\mu} (x) := \frac{d \sigma_{\mu}^{abs}}{d x} (x) = \frac{1}{2 \pi} \int_{x^2/4}^{\infty} \frac{\sqrt{4 v - x^2}}{v} ~ d \nu (v)
< \infty\,,
\end{equation}
for $0 < |x| < \infty$\,.
The even function $\rho_{\mu}$ is decreasing with $|x|$.
\end{proposition}

The proof of this proposition is elementary, essentially an application of Fubini's theorem
to the integrals $\int_{0}^{\infty} \int_{- \infty}^s d \sigma_v (x) ~  d \nu (v)$.
Note that the monotonicity of $\rho_{\mu}$ is obvious from the definition.
It is also the reason why the finiteness stated in \eqref{eq:Tex.5} can not only be shown
to hold for almost all $x \in \mathbb{R} \setminus \{ 0 \}$ (by Fubini), but for all of them.

\begin{remark}\label{rem:Tex.1}
\emph{Spin case.} Observe first that in the spin case the support of the measure $\sigma_{\mu}$ is contained in $[-2,2]$ since the variances
$v(\tau) = 1 - \tau^2$ never exceed the value $1$. Moreover, relation \eqref{eq:B2.1} immediately leads to an expression for $\sigma_{\mu}$
directly in terms of the de Finetti measure $\mu$. Set $a(x) := \sqrt{4-x^2}$, then $\sigma_{\mu} = \mu (\{-1,1\}) \delta_0 + \sigma_{\mu}^{abs}$ with
\begin{equation}\label{eq:Tex.6}
\frac{d \sigma_{\mu}^{abs}}{dx} (x) = \frac{1}{2\pi} \int_{-a (x)/2}^{a (x)/2} \frac{\sqrt{a (x)^2 - 4 t^2}}{1 - t^2} ~  d \mu (t)\,,
\text{~~ for ~ } 0 < |x| \leq 2\,.
\end{equation}
One may evaluate \eqref{eq:Tex.6} explicitly in special cases. For example, if the de Finetti measure $\mu$ equals the uniform distribution $\mu_{\text{uni}}$ on $[-1,1]$,
i.e.~if we have $d \mu_{\text{uni}} (t) := \frac{1}{2} \mathcal{X}_{[-1,1]} (t)  dt$, then the corresponding limiting spectral measure is given by
$d \sigma_{\mu_{\text{uni}}} (x) = \frac{1}{4} (2 - |x|)_+ ~ dx$.
\end{remark}

It is obvious from \eqref{eq:B2.1} that $\sigma_{\mu}$ is a semicircle if  $v(\tau)=a$ for some $a \geq 0$ almost surely w.r.t. the measure $\mu$.
We conclude the present section by showing that the converse is also true:
\begin{equation}\label{eq:Tex.7}
\sigma_{\mu} \, \mbox{ is a semicircle } \; \Leftrightarrow \; \mbox{ there exists } a \in [0, \infty)\,\mbox{ such that } \nu = \delta_a
\end{equation}
We only need to consider "$\Rightarrow$". Suppose that  $\sigma_{\mu}=\sigma_{s}$ for some $s \in [0,\infty)$.
Then all moments of $\sigma_{\mu}$ and $\sigma_{s}$ agree. By linear scaling
\begin{equation*}
\label{eq:B1.5}
m_v^{(k)} := \int x^k d\sigma_v(x) = v^{k/2} m_1^{(k)} \quad \mbox{for all } v \in [0,\infty)\,, k >0 \,.
\end{equation*}
An application of Fubini's theorem gives $\int x^k d\sigma_{\mu}(x) = m_1^{(k)}\int_{0}^{\infty} v^{k/2} d\nu(v)$. The equality of the fourth and second moments of $\sigma_{\mu}$ and $\sigma_{s}$ then yields
\begin{equation*}
\int_{0}^{\infty} v^2 d\nu(v) = s^2 = \left( \int_{0}^{\infty} v d\nu(v)\right)^2\,.
\end{equation*}
This implies that the Cauchy-Schwarz inequality $(\int f d\nu)^2 \leq \int f^2 d\nu$ is an equality for $f(v) = v$. Hence the identity $f$ is a  constant function in $L^2(d\nu)$ proving the claim.

\section{The operator norm for band random matrices}

The semicircle law for Wigner band ensembles suggests that in the case of centred entries
the operator norm should asymptotically be of the order of the square root of the bandwidth $w_N$.
It was already observed in \cite{BMP}
%\marginpar{Zitat ok?}
that this cannot hold if the bandwidths do not grow at
least at some logarithmic rate with the matrix size. In the first subsection we provide in
Theorem \ref{theorem:W1} and in Remark \ref{remark:Pastur} positive results in this
direction that guarantee for centred Wigner band ensembles an almost sure upper bound
on the operator norm that grows proportionally with $\sqrt{w_N}$ if the bandwidth satisfies some growth condition.
The second subsection considers the situation of Wigner band ensembles with arbitrary means and
de Finetti band ensembles.

\subsection{Centred Wigner band ensembles}
\label{sec:WC}

The method of moment was used in \cite{BaiY} to obtain the almost sure limit of the
appropriately rescaled operator norms for centred Wigner ensembles, see also \cite{FuerediK}.
%\marginpar{Zitate ok?}
We follow the basic strategies
of \cite{BaiY} in the form presented in \cite{Tao} and extend the arguments to band matrices.
For the convenience of the reader we present all details but we refer
her or him to \cite[Section 2.3]{Tao} for detailed motivation. A technical but essential lemma
that provides bounds on the expected values of traces of matrix powers for exponents that might
grow with the matrix dimension (see Lemma \ref{lemma:MB}) is deferred to the Appendix.

\begin{theorem}
\label{theorem:W1}
Let $(W_N)_N$ and $(W^{\text{per}}_N)_N$ be centred strict or periodic Wigner band ensembles with variance $v$ as introduced
in Definition \ref{def:eeebm}. Suppose furthermore that there exist positive constants $c$ and $q$ such that the
corresponding bandwidths $w_N$ satisfy the growth condition $w_N \geq c N^q$.

I. In both cases $X_N= W_N$ and $X_N = W^{\text{per}}_N$ we have
\begin{equation}
\label{eq:W.5}
\IP \left(\underset{N \rightarrow \infty}{\limsup} \frac{||X_N||_{op}}{\sqrt{w_N}}
\leq 2\sqrt{v}\right) = 1\,.
\end{equation}

II. In the case of periodic ensembles $X_N= W^{\text{per}}_N$
we obtain the stronger result
\begin{equation}
\label{eq:W.10}
\IP \left(\underset{N \rightarrow \infty}{\lim} \frac{||X_N||_{op}}{\sqrt{w_N}}
= 2 \sqrt{v}\right) = 1\,.
\end{equation}

III. Result \eqref{eq:W.10} also holds for strict band matrices $X_N=W_N$ with half-widths $b_N$ satisfying $\lim_{N \to \infty} \frac{b_N}{N} \in \{0; 1\}$.
\end{theorem}
\begin{remark}
\label{remark:Pastur}
%\marginpar{Zitat ok?\\Hattest Du nicht eine L'Hospital-Rechnung, die es uns erlaubt, etwas genaueres zu sagen?}
It has already be shown in \cite{BMP} that $||X_N||_{op}/\sqrt{w_N}$ is almost surely unbounded
if the bandwidth $w_N$ does not grow at least at some logarithmic rate.

Theorem \ref{theorem:W1} provides a positive result by specifying
a minimal growth rate for the sequence of bandwidths that guarantees almost surely that the scaled operator norms
$||X_N||_{op}/\sqrt{w_N}$ remain bounded. Following the steps in the proof of Theorem
\ref{theorem:W1} shows that the growth condition on the bandwidths is intimately
connected to the decay of the tail of the law for the matrix entries. Recall that
in the statement of Theorem \ref{theorem:W1} this decay is implicitly given by our
general assumption that all moments of the matrix entries are finite. Less decay leads to
stronger conditions on the growth of the bandwidths. Let us assume, in the opposite direction, that
the distribution of the matrix entries has compact support. Then the
truncation procedure in the proof of Theorem \ref{theorem:W1} is not needed and
Lemma \ref{lemma:fundamental_estimate} immediately yields the growth condition
$w_N \geq c (\log N)^{14 + \epsilon}$ on the bandwidths where $c$ and $\epsilon$ can be any
positive constants.
\end{remark}
{\bf Proof of Theorem \ref{theorem:W1}.}\quad
Multiplying the matrices of the ensemble by the factor $1/\sqrt{v}$ and treating the
trivial case $v=0$ separately, we may restrict ourselves to the case $v=1$.\\
 {\em Part I.} \,
Fix $\delta > 0$. By the Borel-Cantelli Lemma it suffices to show
\begin{equation}
\label{eq:W.15}
\sum_{N=1}^{\infty} \IP \Big(||X_N||_{op} \geq (2 + \delta) \sqrt{w_N}\Big) < \infty \,.
\end{equation}
Set $\tilde{K}_N := N^{\alpha}$ with any exponent $0 < \alpha < q/2$
and define a truncated version of the
ensemble $(X_N)_N$ by
\begin{equation*}
\tilde{Y}_N (i,j) := X_N (i,j) \cdot 1_{ \{ |X_N (i,j) | ~  \leq \tilde{K}_N \}} \,.
%\text{~  for all ~} 1 \leq i \leq j \leq N
\end{equation*}
Observe that $(\tilde{Y}_N)_N$ might not be an auxiliary Wigner ensemble AWE (see Definition \ref{def:AWE} in the Appendix),
because its entries are not necessarily centred. A second modification is therefore needed.

The expectation $E_N := \IE (\tilde{Y}_N)$ is an $N \times N$ matrix with entries that take only the values $0$ or  $\IE (\tilde{Y}_N(1,1))$.
Then $e_N:=|\IE (\tilde{Y}_N(1,1))|$ defines an upper bound on the modulus of all entries of $E_N$.
Finally, define  $Y_N := \tilde{Y}_N - E_N$. Clearly all entries of $Y_N$ are centred and we have
\begin{equation*}
 \IE (Y_N^2 (i,j)) \hspace{-1pt}= \hspace{-1pt}\mathbb{V}(Y_N(i,j))
 \hspace{-1pt} = \hspace{-1pt} \mathbb{V}(\tilde{Y}_N(i,j))
 \hspace{-1pt}\leq \hspace{-1pt}\IE (\tilde{Y}_N^2 (i,j)) \hspace{-1pt}\leq \hspace{-1pt}
 \IE (X_N^2 (i,j))\hspace{-1pt}
 = \hspace{-1pt}1
\end{equation*}
for all $1 \leq i \leq j \leq N$ so that condition (C1) of Definition \ref{def:AWE}
is satisfied for $(Y_N)_N$. It is then clear that $(Y_N)_N$ is an AWE with support bounds
$K_N:=\tilde{K}_N+e_N$ and maximal row occupancies $n_N:=w_N$.

Next we derive
a bound on the entries of the matrix $E_N$. Here we use our assumption that all moments
exist so that $C_p := \int |x|^p d\rho(x) < \infty$ for all $p > 0$ where $\rho$ denotes the common
distribution of the matrix entries $X_N(i, j)$. We prove for all $p\geq 1$ that
\begin{equation}
\label{eq:W.20}
e_N \leq C_p \tilde{K}_N^{-(p-1)} = C_p N^{-\alpha (p-1)}\,.
\end{equation}
Indeed, since all entries $X_N(i,j)$ are centred we have
\begin{equation*}
\IE \left( X_N (1, 1) \cdot 1_{ \{ |X_N (1,1) | ~  \leq ~ \tilde{K}_N \} } \right)
= -  \IE \left (X_N (1,1) \cdot 1_{ \{ |X_N (1,1) | ~  > ~ \tilde{K}_N \} } \right)
\end{equation*}
and \eqref{eq:W.20} follows by standard arguments. Using \eqref{eq:W.20} together with our
choice $0 < \alpha < q/2$ and together with the growth condition on the bandwidth, we
may establish
hypothesis \eqref{eq:C3.2} of Lemma \ref{lemma:fundamental_estimate} and we learn
\begin{equation*}
\sum_{N=1}^{\infty} \IP \Big(||Y_N||_{op} \geq \Big(2 + \frac{\delta}{2}\Big) \sqrt{w_N}\Big)
< \infty \,.
\end{equation*}
Claim \eqref{eq:W.15} then follows, provided we can show
\begin{eqnarray}
\label{eq:W.25}
\sum_{N=1}^{\infty} \IP \Big(||Y_N-\tilde{Y}_N||_{op} \geq \frac{\delta}{2} \sqrt{w_N}\Big)
&<& \infty \quad \text{~ and ~}
\\
\label{eq:W.30}
\sum_{N=1}^{\infty} \IP (X_N \neq \tilde{Y}_N)
&<& \infty \,.
\end{eqnarray}
The first estimate \eqref{eq:W.25} follows from \eqref{eq:W.20} in the case $\alpha (p-1) > 1/2$ via
\begin{equation*}
||E_N||_{op} \leq \Big( \sum_{i, j = 1}^{N}  E_N(i,j)^2 \Big)^{1/2} \leq
C_p N^{-\alpha (p-1)} (N w_N)^{1/2}
\end{equation*}
which implies that only finitely many terms in the sum in \eqref{eq:W.25} do not vanish.

Finally, we turn to statement \eqref{eq:W.30}. Markov's inequality gives
\begin{equation*}
\IP \left( |X_N(i,j)|  ~  > ~ \tilde{K}_N  \right)  \leq C_p \tilde{K}_N^{-p} = C_p N^{-\alpha p}\,.
\end{equation*}
Choosing $p > 3/\alpha$ and using the independence of the entries of $X_N$ we have
\begin{equation*}
\IP \left( \tilde{Y}_N \neq X_N \right) \leq \frac{N(N+1)}{2} C_p N^{-\alpha p} =
\mathcal{O}_p \left( N^{2-\alpha p} \right) \,.
\end{equation*}
and the summability claimed in \eqref{eq:W.30} is proved.
Observe that the conditions $p \geq 1$ and $\alpha (p-1) > 1/2$ that we assumed
in our arguments above
are weaker than $p > 3/\alpha$ since $0 < \alpha < q/2 \leq 1/2$.\vspace{3pt}\\
{\em Parts II and III.} \, All the cases considered in parts II and III of Theorem \ref{theorem:W1} have the common feature
that the empirical eigenvalue measures $\mu^{\omega}_{N}$
        of $X_N/\sqrt{w_{N}}$ converge weakly almost surely
        to the semicircle  $\sigma_{v=1}$ (see Theorem \ref{thm:wignerband} b)
        and recall that we have restricted ourselves to the case of unit variance $v$).
        Thus for all bounded continuous functions $f: \mathbb{R}
\to \mathbb{R}$:
\begin{equation}
\label{eq:W.35}
 \IP \Big( \lim_{N \to \infty} \int f d\mu^{\omega}_N = \int_{-2}^2
 f  d\sigma_{v=1}   \Big) = 1 \,.
\end{equation}
Since we have already proved an upper bound in part I it suffices to show that
\begin{equation}
\label{eq:W.40}
\IP \left(\underset{N \rightarrow \infty}{\liminf} \frac{||X_N||_{op}}{\sqrt{w_N}}
< 2 - \delta \right) = 0
\end{equation}
for all small but fixed $\delta > 0$. Choose a continuous functions $f_{\delta}: \mathbb{R}
\to [0, 1]$ that takes the value $1$ on
$\mathbb{R} \setminus (-2; 2)$ and the value $0$ on $[-2 + \delta; 2 - \delta]$. In order to verify
\eqref{eq:W.40} it is enough to convince ourselves that
\begin{equation*}\label{eq:W.45}
\left\{  \underset{N \rightarrow \infty}{\liminf} \frac{||X_N||_{op}}{\sqrt{w_N}} < 2 - \delta
\right\} \subset
\left\{   \lim_{N \to \infty} \int f_{\delta} d\mu^{\omega}_N \neq
\int f_{\delta} d\sigma_{v=1}\right\} \,,
\end{equation*}
where the inequality on the right hand side could also mean that the limit does not exist.
The above inclusion can be seen as follows:

If $\liminf_{N \to \infty} ||X_N||_{op}/\sqrt{w_N} < 2 -  \delta$ then there exists a subsequence
$(N_k)_k$ such that the supports of the empirical measures  $\mu^{\omega}_{N_k}$ are all
contained in the set $[-2 + \delta; 2 - \delta]$ and therefore $\int f_{\delta} d\mu^{\omega}_{N_k} = 0$.
However, the explicit definition of the standard semicircle $\sigma_{v=1}$ implies
$ \int f_{\delta}(x) d\sigma_{v=1}(x) > 0$  and we cannot have
$\int f_{\delta} d\mu^{\omega}_N$ $\to$ $\int f_{\delta} d\sigma_{v=1}$ as
$N \to \infty$.
\hfill \rule{0.5em}{0.5em}\vspace{3pt}
%\hfill \rule{0.5em}{0.5em}

Recall from the proof of Proposition \ref{prop:conv_op'} that we still have to prove that the
forth largest singular value $s_{N-3}$ of a centred full Wigner ensemble with variance $v$
behaves asymptotically like the operator norm, i.e.~like the largest singular value $s_N$, and
converges almost surely to $2\sqrt{v}$ when divided by $\sqrt{N}$. We now prove this claim
in the more general setting of centred Wigner band matrices
for which the semicircle law holds. Therefore the ensembles listed
in Parts~II and III of Theorem \ref{theorem:W1}, and in particular full matrices, are included.

More precisely, we show:
\begin{corollary}\label{cor:s(N-m)}
Parts II and III of Theorem \ref{theorem:W1} also holds true if we replace \eqref{eq:W.10} by
\begin{equation*}
\IP \left(\underset{N \rightarrow \infty}{\lim} \frac{s_{N-m}(X_N)}{\sqrt{w_N}}
= 2\sqrt{v} \right) = 1\,.
\end{equation*}
for any fixed $m \in \mathbb{N}$.
\end{corollary}
\begin{proof}
Since $s_{N-m} \leq s_N$ and $s_N$ equals the $\ell_2$-operator norm the upper
bound follows from Theorem \ref{theorem:W1}. For the lower bound we proceed
exactly as in the proof of Parts~II and III of Theorem \ref{theorem:W1}. Clearly, we can again
restrict ourselves to the case of unit variance $v$. Then the arguments
there show for any fixed $\delta > 0$ that
\begin{equation*}\label{eq:W.50}
\left\{  \underset{N \rightarrow \infty}{\liminf} \, \frac{s_{N-m}(X_N)}{\sqrt{w_N}} < 2 - \delta
\right\} \subset
\left\{   \lim_{N \to \infty} \int f_{\delta} d\mu^{\omega}_N \neq
\int f_{\delta} d\sigma_{v=1} \right\}.
\end{equation*}
Thus, by the semicircle law,
\begin{equation*}
\label{eq:W.55}
\IP \left(\underset{N \rightarrow \infty}{\liminf} \, \frac{s_{N-m}(X_N)}{\sqrt{w_N}}
< 2 - \delta \right) = 0
\end{equation*}
and the claim follows.
\end{proof}

\subsection{Arbitrary means and the de Finetti case}
\label{ssec:4.2}

The proof of Proposition \ref{prop:conv_op} and statement b) of Lemma \ref{lemma:G1} suggest that
the operator norm of non-centred Wigner band ensembles is asymptotically proportional to the
bandwidth $w_N$ rather than to its square root as in the centred case. We formulate this in Theorem
\ref{theorem:F1} immediately for de Finetti band ensembles and remark thereafter that this includes the case
of Wigner band ensembles with arbitrary means.

\begin{theorem}
\label{theorem:F1}
Let $(F_N)_N$ and $(F^{\text{per}}_N)_N$ be strict and periodic de Finetti band ensembles respectively
(see Definition \ref{def:dfbm}). Suppose furthermore that there exist positive constants $c$ and $q$ such that the
corresponding bandwidths $w_N$ satisfy the growth condition $w_N \geq c N^q$.
Denote by $(\xi_i)_i$ the sequence of exchangeable random variables from which
the entries of the matrices are drawn.
Recall also the definition of the empirical mean $M$ in Definition~\ref{def:empirical}.

I. In both cases $X_N= F_N$ and $X_N = F^{\text{per}}_N$ we have
\begin{equation}
\label{eq:F.5}
\IP \left(\underset{N \rightarrow \infty}{\limsup} \frac{||X_N||_{op}}{w_N}
\leq | M(\omega) | \right) = 1 \quad \text{and} \quad \IP \left(\underset{N \rightarrow \infty}{\liminf} \frac{||X_N||_{op}}{w_N}
\geq \frac{3}{4}| M(\omega) | \right) = 1\,.
\end{equation}

II. In the case of periodic ensembles $X_N= F^{\text{per}}_N$
we obtain the stronger result
\begin{equation}
\label{eq:F.10}
\IP \left(\underset{N \rightarrow \infty}{\lim} \frac{||X_N||_{op}}{w_N}
= | M(\omega) | \right) = 1 \,.
\end{equation}

III. Result \eqref{eq:F.10} also holds for strict band matrices $X_N=F_N$ with half-widths $b_N$ satisfying $\lim_{N \to \infty} \frac{b_N}{N} \in \{0; 1\}$.

\end{theorem}
\begin{proof}
{\em Part I.}  Let $\mathbb{P}$ be given as in \eqref{eq:Finetti}. With respect to the probability measure $P_{\Lambda_{\tau}}$
the ensembles $(F_N-m_1(\tau) B_N)_N$ / $(F^{\text{per}}_N-m_1(\tau) P_N)_N$  are centred
strict/ periodic Wigner band matrices with variance $v(\tau)$ (Definition  \ref{def:eeebm}).
We conclude from Theorem \ref{theorem:W1} and  Proposition \ref{prop:exchangeable_mean} that
\begin{equation*}
P_{\Lambda_{\tau}} \left(\underset{N \rightarrow \infty}{\limsup} \frac{||F_N-m_1(\tau) B_N||_{op}}{\sqrt{w_N}}
\leq 2\sqrt{v(\tau)}\right) = 1
\end{equation*}
and
\begin{equation*}
P_{\Lambda_{\tau}} \left(\underset{N \rightarrow \infty}{\limsup} \frac{||F^{\text{per}}_N-m_1(\tau) P_N||_{op}}{\sqrt{w_N}}
\leq 2\sqrt{v(\tau)}\right) = 1
\end{equation*}
%Integrating the two identities above over $\tau$ (see \eqref{eq:Finetti}) we derive that they remain true
%if  is replaced by .
%, and
% Proposition \ref{prop:exchangeable_variance
We now argue that this suffices to show \eqref{eq:F.5}.
 Let us begin with the upper bound. Using  inequality $||X_N||_{op} \leq ||X_N - m_1(\tau) A ||_{op} + | m_1(\tau) | \, ||A||_{op}$ for $A \in \{B_N, P_N\}$,
$||A||_{op} \leq w_N$ by Lemma \ref{lemma:G1}b), and $w_N \to \infty$ for $N \to \infty$ gives the first statement of
\eqref{eq:F.5} with $\mathbb{P}$ being replaced by $P_{\Lambda_{\tau}}$. Integration over $\tau$ with respect to the de Finetti
measure $\mu$ then yields the first relation of \eqref{eq:F.5}.

For the lower bound we proceed in a similar fashion. Lemma \ref{lemma:G1}b) implies $||A||_{op} \geq  \frac{3}{4}w_N$
for $A \in \{B_N, P_N\}$. Together with the general inequality $||X_N||_{op} \geq - ||X_N - m_1(\tau) A ||_{op} + | m_1(\tau) | \, ||A||_{op}$
and with $w_N \to \infty$ integration with respect to the de Finetti measure completes the proof of Part~I.\vspace{3pt}\\
{\em Parts II and III.} The additional assumptions of Parts~II and III imply that $||A||_{op} / w_N$ tends to~$1$ as $N \to \infty$
for $A \in \{B_N, P_N\}$ by Lemma \ref{lemma:G1}b).
This allows to remove the factor $\frac{3}{4}$ in the second statement of
\eqref{eq:F.5} and claim \eqref{eq:F.10} follows.
\end{proof}
\begin{remark}\label{rem:F1}
\begin{enumerate}
\item Strict and periodic Wigner band ensembles are also strict and periodic de Finetti ensembles respectively, and
Theorem~\ref{theorem:F1} applies with the empirical mean $M$ in the relations
of \eqref{eq:F.5} and \eqref{eq:F.10} being replaced by
the mean $m$ of the Wigner ensemble. Indeed,
let $\rho$ denote the law of the entries of the Wigner ensemble. Then $T=\{0\}$, $\Lambda(0) = \rho$, and
$\mu = \delta_0$ provide a suitable representation \eqref{eq:Finetti}. Moreover, the mean $m$
of the Wigner ensemble agrees with the empirical mean $M$ of the corresponding de Finetti ensemble that is deterministic
in this trivial case.
\item Full de Finetti ensembles $(X_N)_N$ may be viewed as periodic de Finetti band ensembles with bandwidth $w_N=N$.
In this case one can show for the second largest singular value (cf. Proposition \ref{prop:conv_op'})
%with empirical variance $V$
\begin{equation} \label{eq:DF_s(N-1)}
\mathbb{P}\Big(\lim_{N \to \infty}\frac{||X_{N}||_{op}^{'}}{\sqrt{N}} = 2 \sqrt{V(\omega)} \Big) = 1
\end{equation}
where $V$ denotes again the empirical variance.
Thus, whenever the empirical mean $M$ does not vanish on a set $\Omega' \subset \Omega$ of positive probability
the discrepancy on $\Omega'$ between the growth of the operator norm
(order $N$, see Theorem~\ref{theorem:F1}) and the growth given by the semicircle law
(order $\sqrt{N}$, see Theorem~\ref{thm:DFB-semicircle}) is caused by a single outlier.

In order to prove \eqref{eq:DF_s(N-1)} it suffices to show this relation with $\mathbb{P}$ being replaced by $P_{\Lambda_{\tau}}$
and $V(\omega)$ being replaced by $v(\tau)$ for any $\tau \in T$. This, however, is exactly the assertion of Proposition \ref{prop:conv_op'}.
\end{enumerate}
\end{remark}

\appendix
\section{Moment method for Wigner ensembles}

In the main text we make use of two results, Theorem~\ref{thm:wignerband} b) in the centred case and Lemma~\ref{lemma:fundamental_estimate}, that are essentially
known and can be proved using the method of moments. Since theses proofs are not readily available in the literature
we provide them in the subsequent subsections for the convenience of the reader.

The method of moments is based on the fact that a large class of probability measures $\sigma$, including in particular all such
measures with compact support, are determined by their sequence of moments $(\int x^k $d$\sigma(x))_{k \in \mathbb{N}}$.
Since the sum of the $k$-th powers of all eigenvalues of some $N\times N$ matrix $X$ is given by the trace of $X^k$ it is clear that
the $k$-th moment of the eigenvalue distribution measure of $X$ is given by $\frac{1}{N}$tr$(X^k)$.
Moreover, it follows for matrices $X$ with real spectrum that for even positive integers $k$ the moduli of all eigenvalues are bounded
above by the $k$-th root of tr$(X^k)$. This shows in a nutshell how the objects studied in this paper are related to traces of matrix powers.

Wigner introduced in \cite{Wigner1, Wigner2} a method to analyse the large $N$ asymptotics of the expectations
of traces of matrix powers for ensembles that now bear his name. We begin by deriving a useful representation for these expectations.

Let $(X_N)_N$ be a Wigner ensemble that may be full or banded. Then \vspace{-7pt}
\begin{equation}\label{expansion_trace}
\mbox{tr} (X_N^k) = \sum_{\gamma_1, \ldots , \gamma_k = 1}^N X_N (\gamma_1, \gamma_2)
\cdot \ldots \cdot X_N (\gamma_k, \gamma_1) =
 \sum_{\gamma \in \mathcal{P}} X_{\gamma}\vspace{-3pt}
 \end{equation}
 with $X_{\gamma} := \prod^k_{s=1} X_N (\gamma_s, \gamma_{s+1})$ and
 \begin{equation}\label{set_of_paths}
\mathcal{P} := \{ \gamma \in \{1, \ldots , N \}^{k+1} ~  | ~  \gamma_1 = \gamma_{k+1} \text{~ and  ~} X_{\gamma} \neq 0 \}\,.
\end{equation}
By the condition $X_{\gamma} \neq 0$ we mean that no factor $X_N (\gamma_s, \gamma_{s+1})$ in the definition
of $X_{\gamma}$ is identically equal to 0 due to the prescribed band structure of the matrix (see Definition~\ref{def:eeebm}).
Thus the set $\mathcal{P}$ not only depends on $N$ and $k$ but also on the half-width $b_N$ and on the question whether
strict or periodic band matrices are considered. We would like to alert the reader that none of this required information is recorded
in the notation.
For each path $\gamma \in \mathcal{P}$ we denote
\begin{equation*}
\mathcal{E} (\gamma) := \{ \{ \gamma_s, \gamma_{s+1} \} ~  | ~  s \in \{1, \ldots , k\} \}
\end{equation*}
the set of (undirected) edges,
\begin{equation*}
  \eta (\gamma) := \#\, \mathcal{E} (\gamma)\,,
\end{equation*}
and by
$e_1 (\gamma), \cdots , e_{\eta (\gamma)} (\gamma)$ the elements of $\mathcal{E} (\gamma)$
in their order of appearance as one travels along the path $\gamma$.

Furthermore, for each $1 \leq i \leq \eta (\gamma)$, let $a_i (\gamma) \in \IN$ be the multiplicity
with which edge $e_i (\gamma)$ occurs and set $\xi_i (\gamma) := X_N (e_i (\gamma))$.
Here we use $X_N (e) := X_N (p,q)$ for edges $e = \{p,q \}$ which is well defined
by the symmetry of $X_N$.
%To simplify notation we are going to omit the $\gamma$-dependency from now on most of the time.
The assumed independence of matrix entries in the upper triangular part
gives for all $\gamma \in \mathcal{P}$:
\begin{equation}
\label{eq:C6.1a}
\IE (X_{\gamma}) = \IE \left(\prod_{i=1}^{\eta(\gamma)} \xi_i(\gamma)^{a_i(\gamma)}\right) = \prod^\eta_{i=1} \IE (\xi_i^{a_i}) \,,
\end{equation}
where we have omitted the $\gamma$-dependency in the last term for notational simplicity.
For {\em centred} Wigner ensembles we have in addition that
$\IE (X_{\gamma}) = 0$ if there exists an $i \in \{1, \ldots, \eta \}$ with $a_i = 1$. This, by the way, is the reason why the
method of moments works so well in the case of centred entries and cannot be applied directly for non-vanishing means. Set
$$ \mathcal{P}_0 := \{ \gamma \in \mathcal{P} ~  | ~  a_i \geq 2
\text{~  for all ~} 1 \leq i \leq \eta(\gamma) \} \,.$$
Then, for centred Wigner ensembles
 \begin{equation}
\label{eq:C6.2a}
\IE (\mbox{tr} (X_N^k)) = \sum_{\gamma \in \mathcal{P}_0} \IE (X_{\gamma}) \,.
\end{equation}
Together with formula \eqref{eq:C6.1a} we have a representation for the expectation of the $k$-th moments
of the eigenvalue distribution measures related to $X_N$ that is used in both of the following subsections.

\subsection{Proof of Theorem \ref{thm:wignerband} b) in the centred case}
\label{app:sec1}

Denote by $X_N$ a {\em centred} Wigner band ensemble, periodic or strict, that satisfies the corresponding assumptions on the
sequence $(b_N)_N$ of half-widths stated in Theorem \ref{thm:wignerband}. Recall the notation introduced at the beginning of
this appendix and denote for any path $\gamma \in \mathcal{P}_0$ by $r(\gamma)$ the number of different vertices contained
in the path. Since every edge occurs at least twice the number of different edges $\eta (\gamma)$ is bounded
above by $\frac{k}{2}$ and consequently $r(\gamma) \leq 1 +\frac{k}{2}$. The following asymptotics on the number of elements
in $\mathcal{P}_0$ with a prescribed number of vertices $1\leq r \leq 1+\frac{k}{2}$ use all of the assumptions of Theorem \ref{thm:wignerband}
on the bandwidths (recall also the remark after definition \eqref{set_of_paths}):
\begin{equation}\label{asympt:n_rk}
n_{r,k}(N) := \# \{\gamma \in \mathcal{P}_0 ~  | ~  r (\gamma) = r \} \sim c_{r, k} N w_N^{r-1}
\end{equation}
as $N \to \infty$. Here $\sim$ means that the ratio of left hand side and right hand side converges to $1$. The number $c_{r, k}$ in
\eqref{asympt:n_rk} can be defined as the number of {\em prototypes} in $\{\gamma \in \mathcal{P}_0 ~  | ~  r (\gamma) = r \}$.
By a prototype we understand a path $\gamma$ that has the additional property that the vertices are numbered in the order of
their occurrence in the path. For example, $\gamma=(1,2,3,2,4,2,3,2,1)$ is a prototype in $\mathcal{P}_0$ for $k=8$ with $r(\gamma)=4$.

The asymptotic formula \eqref{asympt:n_rk} indicates that only paths  $\gamma \in \mathcal{P}_0$
with the maximal number of vertices $r(\gamma) = 1 +\frac{k}{2}$ matter.
For such paths, which can only occur for even values of $k$, the
number of different edges is also maximal $\eta (\gamma)= \frac{k}{2}$ and all edges $e_i$ have multiplicity $a_i=2$.
Since the second moments of all matrix entries are assumed to equal
the same constant $v$ \neu we obtain $\IE (X_{\gamma}) = v^{k/2}$. In summary we have argued that
\begin{equation*}
\IE (\mbox{tr} (X_N^k)) \phantom{A}
\begin{cases}
      \phantom{A} \sim c_{1+\frac{k}{2}, k} N (v w_N)^{k/2} &, \text{~  if $k$ is even,} \vspace{5pt}\\
        \phantom{A} ={\mathcal O} \big(N w_N^{(k-1)/2}\big)  &, \text{~  if $k$ is odd.}
        \end{cases}
\end{equation*}
These asymptotics show in particular that one needs to divide $X_N$ by $\sqrt{w_N}$ for the expected
moments of the eigenvalue distribution measures to converge to some nontrivial measure.
This justifies the definition of the corresponding
measures $\mu_N^{\omega}$ in the statement of Theorem \ref{thm:wignerband}. Indeed, we may conclude
for every positive integer $k$ that
\begin{equation*}
\lim_{N\to \infty} \IE \Big( \int x^k d\mu_N^{\omega}(x) \Big) =  \lim_{N\to \infty} \frac{\IE (\mbox{tr} (X_N^k))}{N w_N^{k/2}}
=
\begin{cases}
      c_{1+\frac{k}{2}, k} v^{k/2} & \text{for even $k$,} \vspace{5pt}\\
        0  & \text{for odd $k$.}
        \end{cases}
\end{equation*}
It is a classical result in combinatorics (see e.\,g. \cite{Stanley}, Theorem 1.51)
 that the numbers $c_{1+\frac{k}{2}, k}$ ($k$ even) are
given by Catalan numbers $C_{\frac{k}{2}}=\frac{1}{1+k/2}\,\binom{k}{k/2}$ and that
\begin{equation}\label{conv:expected_moments}
\lim_{N\to \infty} \IE \Big( \int x^k d\mu_N^{\omega}(x) \Big) =  \int x^k d\sigma_v(x)
\end{equation}
holds for all $k$. In order to prove part b) of Theorem \ref{thm:wignerband} it suffices to enhance \eqref{conv:expected_moments}
to $\IP$-almost sure convergence which then implies that $\mu_N^{\omega}$ converges weakly $\IP$-almost surely to
$\sigma_v$.
%\cite{???}\marginpar{welche(s) Zitat(e) h\"attest Du hier gerne?}.
Standard applications of the Chebyshev inequality and of the Borel-Cantelli lemma show that
this can be achieved by proving for every positve integer $k$ the summability of the variances
 \begin{equation}\label{summability_variances}
\sum_{N=1}^\infty \IV \Big( \int x^k d\mu_N^{\omega}(x) \Big) < \infty\,.
\end{equation}
Using the notation introduced in \eqref{expansion_trace} and \eqref{set_of_paths} we obtain
 \begin{equation}\label{expansion_variance}
\IV \Big( \int x^k d\mu_N^{\omega}(x) \Big) =
\frac{1}{N^2 w_N^k} \sum_{\gamma, \gamma' \in {\mathcal P}}
\IE(X_{\gamma} X_{\gamma'}) - \IE(X_{\gamma}) \IE(X_{\gamma'})
\end{equation}
Observe that pairs $(\gamma, \gamma') \in {\mathcal P}^2$ that do not share a common edge do not contribute to the sum,
because $\IE(X_{\gamma} X_{\gamma'}) = \IE(X_{\gamma}) \IE(X_{\gamma'})$ holds in this case by the assumed independence of the
matrix entries in the upper triangular part of $X_N$. For pairs $(\gamma, \gamma')$ that do share a common edge we construct
a path $\hat{\gamma} = \hat{\gamma}(\gamma, \gamma') \in {\mathcal P}$ with $2k-2$ edges (!) in the following way. Choose $i_0$
such that $e_{i_0}$ is the first edge
in ${\mathcal E}(\gamma)$ that also appears ${\mathcal E}(\gamma')$ and let $i'_0$ be minimal with $e_{i_0}=e_{i'_0}$.
Furthermore, denote $P:=\gamma_{i_0}$ and $Q:=\gamma_{i_0+1}$. By construction we have $P=\gamma'_{i'_0+1}$ or
$P=\gamma'_{i'_0}$ In the first case we set
$$
\hat{\gamma} := (\gamma_{i_0+1}, \gamma_{i_0+2}, \ldots \gamma_k, \gamma_1, \ldots \gamma_{i_0}, \gamma'_{i'_0+2},  \gamma'_{i'_0+3},
\ldots \gamma'_k, \gamma'_1, \ldots \gamma'_{i_0})$$
and
$$
\hat{\gamma} := (\gamma_{i_0+1}, \gamma_{i_0+2}, \ldots \gamma_k, \gamma_1, \ldots \gamma_{i_0}, \gamma'_{i'_0-1},  \gamma'_{i'_0-2},
\ldots \gamma'_1, \gamma'_k, \ldots \gamma'_{i_0+1})$$
in the latter case. Loosely speaking $\hat{\gamma}$ is the path that starts and ends in $Q$ by first connecting $Q$ with $P$ following the cyclic
extension of $\gamma$ and then connecting $P$ back to $Q$ along the cyclic extension of $\gamma'$, adapting the direction if necessary.
Since we are dealing with a centred ensemble only those pairs $(\gamma, \gamma')$ with a common edge have a non-zero contribution to
the sum in \eqref{expansion_variance} for which $\hat{\gamma}(\gamma, \gamma') \in {\mathcal P}_0$. From this we conclude that
the number of non-zero terms in the sum in \eqref{expansion_variance}
is bounded above by
 \begin{equation}\label{bound_non-zero_summands}
k^2 \sum_{r=1}^k n_{r,2k-2}(N) = {\mathcal O}\big( N w_N^{k-1} \big)
\end{equation}
(see \eqref{asympt:n_rk} for a definition of $n_{r,k}(N)$). The prefactor $k^2$ takes into account that the values of $i_0$ and $i'_0$ are lost in the
construction and this is the only reason why the map $(\gamma, \gamma') \mapsto \hat{\gamma}$ is not injective.

Inserting the upper bound \eqref{bound_non-zero_summands} into the representation \eqref{expansion_variance} for the variance, we obtain \vspace{-5pt}
\begin{equation}\label{total_bound}
\IV \Big( \int x^k d\mu_N^{\omega}(x) \Big) =
{\mathcal O} \Big( \frac{1}{N w_N} \Big)\,.
\end{equation}
The assumption $\sum_N (N b_N)^{-1} < \infty$ of Theorem \ref{thm:wignerband} b) therefore implies
the desired summability \eqref{summability_variances}. We mention in passing that for centred ensembles part a) of
Theorem~\ref{thm:wignerband} follows from \eqref{conv:expected_moments} and
$$\lim_{N\to \infty} \IV \Big( \int x^k d\mu_N^{\omega}(x) \Big) = 0$$
which is also a consequence of \eqref{total_bound}.

\subsection{Statement amd proof of Lemma \ref{lemma:fundamental_estimate}}
\label{app:sec2}

The proof of Lemma~\ref{lemma:fundamental_estimate} relies on an estimate on the expected trace of
matrix powers as stated in Lemma~\ref{lemma:MB}. The essential difference from the
analysis of the previous subsection is that we need to allow the exponent to grow with the
matrix dimension $N$.

Both, Lemma~\ref{lemma:fundamental_estimate} and Lemma~\ref{lemma:MB}, are formulated for
an auxiliary type of Wigner ensembles that is convenient for the analysis of
Wigner band matrices treated in Subsection \ref{sec:WC}. These auxiliary
Wigner ensembles arise due to the truncation procedure in the proof of
Theorem \ref{theorem:W1}. The truncation depends on the size $N$ of the matrices.
This is why we cannot insist that all entries of the ensemble are identically distributed.
Consequently, the auxiliary Wigner ensembles do not fall into the class of Wigner ensembles
described in Definitions \ref{def:fullWE} and \ref{def:eeebm}, albeit they are still well embedded
in a more general framework of Wigner ensembles that is well-known in the literature.

As it turns out the results in this subsection do not use the spatial structure of band matrices.
They only require a bound on the maximal number of entries in each row that do not vanish identically.
In order to bring this to the fore we do not require the band structure for the auxiliary ensembles.
This generalisation, however, is not used in the present paper.

The arguments used in the proofs are taken from the monograph \cite[Section 2.3]{Tao} where full matrices are discussed.
The adaption to the case of band matrices does not pose additional difficulties. However, we are
somewhat more careful in the formulation of Lemma~\ref{lemma:MB} since the inequalities there
determine the growth conditions on the bandwidths as discussed in Remark~\ref{remark:Pastur}.
This is also our motivation to improve on inequality \eqref{eq:C13.1} for which we present a detailed proof.

\begin{definition}
\label{def:AWE}
By an Auxiliary Wigner Ensemble AWE we understand a probability measure $\IP$ on families
$(X_N)_N$ of real symmetric $N \times N$-matrices such that sequences $(K_N)_N$, $(n_N)_N$
in $[1, \infty)$ exist for which conditions (C1) - (C3) hold for all $N \in \IN$.
\begin{itemize}
\item[(C1)] The entries $X_N (i,j)$, $1 \leq i \leq j \leq N$, are independent with \\
$\IE (X_N (i,j)) = 0\quad$ and $\quad\IE (X_N (i,j)^2) \leq 1$.
\item[(C2)] $\IP (|X_N (i,j)| \geq K_N) = 0$ \, for all $1 \leq i, j \leq N$.
\item[(C3)] For all $1 \leq i \leq N: \# \{j \in \{1, \ldots, N \} | \IP (X_N (i,j) = 0) < 1\} \leq n_N$.
\end{itemize}
We call $K_N$ the \emph{support bound} and $n_N$ the \emph{maximal row occupancy} of $X_N$.
\end{definition}

Observe that the band matrices introduced in Definition \ref{def:eeebm}, strict or periodic,
satisfy Condition (C3) with $n_N=w_N$ (cf.~Remark \ref{remark:wN}). We are now ready to
state the main results of this subsection.

\begin{lemma}
\label{lemma:MB}
Let $(X_N)_N$ be a AWE with support bounds $(K_N)_N$ and
maximal row occupancies $(n_N)_N$. Then for all integers $k, N \in \IN$  with
%$k$ even and
$2 K^2_N k^{14} \leq n_N$ we have
\begin{equation*}
|\IE (\mbox{tr}(X_N^k))| \leq 4 N (2 \sqrt{n_N})^k \,.
\end{equation*}
\end{lemma}
Before proving Lemma \ref{lemma:MB}
we apply it to the
operator norm. As mentioned in the introduction to the Appendix the connection is based on the observation
that for all even $k \in \IN$ we have
\begin{equation}
\label{eq:C3.1}
||X_N||^k_{op} \leq \mbox{tr} (X_N^k) \,.
\end{equation}

\begin{lemma}
\label{lemma:fundamental_estimate}
Let $(X_N)_N$ be a AWE with support bounds $(K_N)_N$ and
maximal row occupancies $(n_N)_N$. Assume furthermore that
\begin{equation}
\label{eq:C3.2}
\sup_{N \in \mathbb N} \;\frac{1}{n_N} \, K_N^2 (\log N)^{14 + \epsilon} \;< \;\infty \qquad
\mbox{for some } \epsilon > 0\,.
\end{equation}
Then
\begin{equation}
\label{eq:C3.3}
\sum_{N=1}^{\infty} \IP (||X_N||_{op} \geq (2 + \delta) \sqrt{n_N}) < \infty
\end{equation}
for any $\delta > 0$ and it follows from the Borel-Cantelli Lemma that
\qquad \qquad
$$\displaystyle \IP \left(\underset{N \rightarrow \infty}{\limsup} \frac{||X_N||_{op}}{\sqrt{n_N}}
\leq 2\right) = 1\,.$$
\end{lemma}
\begin{proof}
Fix $\delta > 0$.
For even $k \in \IN$ relation \eqref{eq:C3.1} and Markov's inequality yield the estimate
\begin{equation}
\label{eq:C4.1}
\IP (||X_N||_{op} > (2 + \delta) \sqrt{n_N}) \leq \frac{\IE (\mbox{tr} (X_N^k))}{[(2 + \delta) \sqrt{n_N}]^k} \,.
\end{equation}
Assumption \eqref{eq:C3.2} implies the existence of a number $C > 0$
such that for all $N \in \mathbb {N}$:
\begin{equation}
\label{eq:C4.2}
2^{15} K_N^2 (\log N)^{14 + \epsilon} \leq C n_N \,.
\end{equation}
Since $\log N > 1$ for all $N \geq 3$ we may choose even integers $k_N$ satisfying
\begin{equation}
\label{eq:C4.3}
(\log N)^{1 + \frac{\epsilon}{15}} \leq k_N \leq 2 (\log N)^{1 + \frac{\epsilon}{15}}
\end{equation}
for all $N \geq 3$. The upper bound in \eqref{eq:C4.3} together with \eqref{eq:C4.2} yield
the inequality $2 K_N^2 k_N^{14} \leq C (\log N)^{-\frac{\epsilon}{15}} n_N$.
Thus there exists $N_0 \geq 3$ such that the hypothesis $2 K_N^2 k_N^{14} \leq n_N$
of Lemma \ref{lemma:MB} holds for all $N \geq N_0$. Hence the right hand side
of \eqref{eq:C4.1} with $k = k_N$ can be bounded above by
\begin{equation*}
4 N (1 + \delta/2)^{-k_N} \leq 4 N^{1 - (\log N)^{\frac{\epsilon}{15}} \log (1 + \delta/2)}
\end{equation*}
for all $N \geq N_0$, where we have also used the lower bound in \eqref{eq:C4.3}. This proves
\eqref{eq:C3.3}.
\end{proof}

{\bf Proof of Lemma \ref{lemma:MB}.}\quad
Recall the notation introduced at the beginning of the Appendix. We begin by estimating $| \IE (\xi_i^{a_i}) |$
that appears in \eqref{eq:C6.1a}. As we are dealing with centred ensembles we only need to consider
the case of $a_i \geq 2$. It follows from conditions (C1), (C2) of Definition \ref{def:AWE}
that $| \IE (\xi_i^{a_i}) | \leq K_N^{a_i - 2} \IE (\xi_i^{2})\leq K_N^{a_i - 2}$.
As $\sum_{i=1}^{\eta(\gamma)} a_i = k$ equals the total number of steps of the path $\gamma$ we obtain
for each $\gamma \in \mathcal{P}_0$ the bound $|\IE (X_{\gamma}) | \leq K_N^{k - 2\eta}$
from \eqref{eq:C6.1a} (where $\eta=\eta(\gamma)$). Formula \eqref{eq:C6.2a} then proves
\begin{proposition}
\label{prop:C7}
Let $(X_N)_N$ be a AWE with support bounds $(K_N)_N$. Then for positive integers $k$:
\begin{equation*}
\label{eq:C6.3}
| \IE (\mbox{tr} (X_N^k)) | \leq \sum_{j=1}^{\lfloor k/2 \rfloor} K_N^{k - 2j} M_j \,,
\end{equation*}
where $M_j := \# \{\gamma \in \mathcal{P}_0 ~  | ~  \eta (\gamma) = j \}$.
\end{proposition}
%Observe that Proposition \ref{prop:C7} does not make use of condition (C3)
%in Definition \ref{defi:C1} of GWE's.
The main effort of proving Lemma \ref{lemma:MB} is to obtain
combinatorial bounds on the numbers $M_j$.
\begin{lemma}
\label{lemma:CL7}
Let $(X_N)_N$ be a AWE with maximal row occupancies $(n_N)_N$.
For all integers $j, k, N \in \IN$,
%with $k$ even
$1 \leq j \leq \frac{k}{2}$, and $n_N \geq 2k^3$ we have (cf. Proposition \ref{prop:C7})
\begin{equation*}
M_j \leq 2N (2 \sqrt{n_N})^k \left(\frac{k^7}{\sqrt{n_N}}\right)^{k - 2j} \,.
\end{equation*}
\end{lemma}
Assuming the validity of Lemma \ref{lemma:CL7} we may deduce the claim of Lemma \ref{lemma:MB}.
Indeed, since the assumption $2 K_N^2 k^{14} \leq n_N$ of Lemma \ref{lemma:MB}
implies $n_N \geq 2 k^3$ (recall $K_N \geq 1$ from Definition \ref{def:AWE}) we may apply
Lemma \ref{lemma:CL7}. In addition we also have $(k^7 K_N/\sqrt{n_N})^2 \leq \frac{1}{2}$ and
the sum in the statement of Proposition \ref{prop:C7} is dominated by a geometric series,
implying the bound of Lemma \ref{lemma:MB}.
\hfill \rule{0.5em}{0.5em}

We are left to derive the combinatorial estimate of Lemma \ref{lemma:CL7}.

{\bf Proof of Lemma \ref{lemma:CL7}.}\quad
Fix $j \in \{1, \ldots, \lfloor \frac{k}{2} \rfloor \}$ and integers
$a_1, \ldots, a_j \geq 2$ with $\sum_{i=1}^j a_i = k$. The main work goes into proving
 \begin{equation}
\label{eq:C8.1}
M_{a_1, \ldots, a_j} \leq 2 N 2^k n_N^j k^{6 (k - 2j)}, \text{~  with}
\end{equation}
$M_{a_1, \ldots, a_j} := \# \{\gamma \in \mathcal{P}_0 ~| ~  \eta (\gamma) = j$
and $a_i (\gamma) = a_i$ for all $1 \leq i \leq j \}$.
As the bound in \eqref{eq:C8.1} is independent of the values of $a_1, \ldots, a_j$ one obtains the
statement of Lemma \ref{lemma:CL7} by showing that there are at most $k^{k - 2j}$ choices for
$a_1, \ldots, a_j$. This in turn follows e.g.~from the observation that each configuration of
$a_1, \ldots, a_j$ with $a_i \geq 2$ and $\sum_{i=1}^j a_i = k$ is mapped bijectively via
\begin{equation*}
(a_1, \ldots, a_j) \mapsto \Big(\sum_{p=1}^q (a_p - 1)\Big)_{1 \leq q \leq j-1}
\end{equation*}
to a selection of $j-1$ different elements out of $\{1, \ldots, k-j-1 \}$.

For that we have
\begin{equation*}
\binom{k-j-1}{j-1} = \binom{k-j-1}{k-2j} \leq k^{k-2j}
\end{equation*}
possibilities.

The proof of \eqref{eq:C8.1} requires a somewhat involved enumeration procedure. Let us first
introduce some terminology: A path $\gamma \in \mathcal{P}_0$ consists of $k$ steps $(s, s+1)$ with
$s = 1, \ldots, k$. For each such $s$ there exists $i \in \{1, \ldots, j \}$ such that
$\{\gamma_s , \gamma_{s+1} \} = e_i$. We then say that step $s$ is taken along edge $e_i$ and that
this step departs from vertex $\gamma_s$ and arrives at $\gamma_{s+1}$. We call edges $e_i$ to be
of {\em higher multiplicity} iff $a_i \geq 3$ and we denote their number by $l = l (\gamma)$. For the total
number of steps along edges of higher multiplicities $L := \sum_{a_i \geq 3} a_i$ the following relations
hold:
\begin{equation*}
k = \sum_{i=1}^j a_i = 2 (j-l) + L, \text{~  and ~} L \geq 3l \,.
\end{equation*}
These imply the useful estimates
\begin{equation}
\label{eq:C10.1}
l \leq k - 2j, ~  L = k - 2j + 2l \leq 3 (k - 2j) \,.
\end{equation}
Steps are called {\em opening steps} or {\em closing steps} iff they are taken along an edge
of multiplicity 2 for the first or for
the second time respectively. Opening steps are called {\em innovative} iff they arrive at a vertex that hasn't
appeared in the path before.
We denote by $m = m (\gamma)$ the number of non-innovative opening steps.
Unlike $l$ and $L$ the number $m$ is not
determined by $a_1, \ldots, a_j$ and may take values $0 \leq m \leq j-l$.

To derive \eqref{eq:C8.1} we proceed as follows: Besides $a_1 , \ldots, a_j$ fix also the integer $m$.
We want to estimate the number of $\gamma \in \mathcal{P}_0$ with $\eta (\gamma) = j$, $a_i (\gamma) = a_i$
for all $1 \leq i \leq j$, and $m (\gamma) = m$.
In order to obtain our bounds we divide the set of all such paths $\gamma$ into different types.

A type determines at which step
\begin{enumerate}
\item the edge $e_i$ occurs, for all $i$ with $a_i \geq 3$ (higher multiplicity),
\item a non-innovative opening step occurs,
\item an innovative opening step occurs.
\end{enumerate}
A crude upper bound on the number of different types is given by (see also \eqref{eq:C10.1})
\begin{equation*}
k^L \cdot k^m \cdot \binom{k}{j-l-m} \leq k^{3 (k - 2j) + m} \cdot 2^k
\end{equation*}
Below we argue that each type contains at most
\begin{equation}
\label{eq:C11.1}
N \cdot k^{3 (k - 2j)} \cdot n_N^{j-m} \cdot k^{2m}
\end{equation}
different paths, so that
\begin{equation*}
M_{a_1, \ldots , a_j} \leq N 2^k k^{6 (k - 2j)} n_N^j \sum_{m=0}^{j-l} \left(\frac{k^3}{n_N}\right)^m
\end{equation*}
and \eqref{eq:C8.1} follows due to the assumption $2k^3 \leq n_N$.

Thus we are left to establish the upper bound \eqref{eq:C11.1} on the number
of paths of any given type, i.e.~on the number of possibilities to choose vertices.
Let us proceed inductively along the path so that only the starting point and the vertices of arrival
need to be selected for each step. The number of possibilities at every step depends
on the kind of step being taken. We distinguish five cases:

\underline{Case A}: Starting point.

There are $N$ possibilities.

\underline{Case B}: Steps along edges of higher multiplicitiy.

If an edge $e_i$, with $a_i \geq 3$ is taken for the first time then there are at most $n_N$ possibilities
to choose the vertex of arrival. Otherwise this vertex is already determined uniquely. In total we have at most
$n_N^l$ possibilities from all steps of this kind.

\underline{Case C}: Non-innovative opening steps.

Since we may only select a vertex of arrival that has already appeared in the path,
the number of choices for each such step is crudely bounded by $k$, in total by $k^m$.

\underline{Case D}: Innovative opening steps.

At each of these steps we have at most $n_N$ possibilities to select the vertex of arrival,
in total at most $n_N^{j-l-m}$ choices.

\underline{Case E}: Closing steps.

Here a subtle difficulty occurs. Note that the type of the path does not determine
at which step a particular edge $e_i$ of multiplicity $2$ is being closed (unlike in the case of
edges of higher multiplicity). The type prescribes only which steps are closing steps. However,
given the choice of vertices up to the closing step we may only perform a step along an
open edge, i.e.~along an edge of multiplicity 2 which has got a previous
opening step but not a previous closing step.
In case there is at most one such open edge that contains our vertex of departure then there
is at most one possible choice for picking the vertex of arrival. Denote by $f$ the number of instances
out of the $j-l$ closing steps for which more than one choice exists for selecting the vertex of arrival.
For each such instance that we call a { \em free closing step} we use the crude bound $k$ on the
 number of choices. Moreover, we show below that the number of free closing steps is bounded by
$f \leq L + m \leq 3 (k - 2j) + m$ (see \eqref{eq:C10.1})
so that the total bound reads $k^{3 (k - 2j) + m}$.

Multiplication of the estimates from all five Cases A-E leads to \eqref{eq:C11.1}.

The proof of Lemma \ref{lemma:CL7} is thus concluded by showing
\begin{equation}
\label{eq:C13.1}
f (\gamma) \leq L (\gamma) + m (\gamma) \text{~  for all ~} \gamma \in \mathcal{P}_0 \,.
\end{equation}
In order to derive \eqref{eq:C13.1} we count for each vertex $b$ on the path the number $f(b)$
of closing steps that depart from $b$ and for which more than one possibility exists to do so.
Denote furthermore by $L(b)$ the total number of steps arriving at $b$ along an edge of higher multiplicity
and by $m(b)$ the total number of steps arriving at $b$ along a non-innovative opening step.
It suffices to show
\begin{equation}
\label{eq:C13.2}
f (b) \leq L (b) + m (b) \,,
\end{equation}
because summation over all vertices $b$ that appear in the path $\gamma$
then proves \eqref{eq:C13.1}.

Our proof of \eqref{eq:C13.2} requires to distinguish 8 cases. The reasoning in all these cases
is quite similar and we begin by presenting it in the simplest case that the loop $\{ b \}$ is not an edge
of the path $\gamma \in \mathcal{P}_0$  and that $b$ is not the starting point of $\gamma$.
We denote by $\Lambda_s(b,\gamma)$ the number of edges
containing $b$ that are open after step $s$
has been completed, i.e.~the number of edges of $\gamma$ of
multiplicity 2 that contain $b$, that are opened at one of the steps $1, \ldots, s$, and that
are closed at one of the steps $s+1, \ldots, k$. Observe that \eqref{eq:C13.2} is trivially
satisfied if $f(b)=0$. Otherwise pick $t_0 \in \{ 1, \ldots, k \}$ so that step $t_0$ is the last
free closing step departing from $b$. We argue below that
\begin{equation}
\label{eq:C13.3}
1 \leq \Lambda_{t_0}(b,\gamma) \leq 2(m(b) +1) + L(b) - 2f(b)
\end{equation}
from which \eqref{eq:C13.2} follows using $\lfloor \frac{1}{2} (L(b)+1) \rfloor \leq L(b)$.

The first inequality of \eqref{eq:C13.3} is a consequence of our definition that step $t_0$ is a
free closing step. The second inequality is based on the observation that
$\Lambda_s(b,\gamma)$ may only change its value at steps $s$ that are connected to
visits of the path $\gamma$ at the vertex $b$. Since we assumed that the loop
$\{ b \}$ is not an edge of $\gamma$ each visit consists of a step of arrival at $b$ and a subsequent
step of departure. For each visit we obtain an upper bound on the change of the value of
$\Lambda_s(b,\gamma)$ by assuming that the step of departure is an opening step.
In case the step of arrival is an opening step/ a step along an edge of higher multiplicity/
a closing step this upper bound is given by 2/1/0 respectively. Using the assumption that $b$
is not the starting point of $\gamma$, the fact that there are at most $m(b)+1$ visits to $b$ that
arrive by an opening step, and the fact that there are at most $L(b)$ visits to $b$ that
arrive along an edge of higher multiplicity we have derived the second inequality of \eqref{eq:C13.3}
except for the term $-2 f(b)$. This term is explained by the observation that we have overestimated
the change of the value of $\Lambda_s(b,\gamma)$ by $2$ whenever the departure from $b$ is
realized by a closing step. In addition, we know that there must be at least $f(b)$
such instances up to step $t_0$.

For a complete proof of \eqref{eq:C13.2} we distinguish

{\bf Case 1:} $b$ is not the starting point of $\gamma$

{\bf Case 2:} $b$ is the starting point of $\gamma$

Both cases are divided into four subcases each:

{\bf Case A:} $\{ b \} \notin \mathcal{E} (\gamma)$ or multiplicity of $\{ b \} \geq 3$

{\bf Case B:} Multiplicity of $\{ b \} = 2$ and closing step of $\{ b \}$ is not free

{\bf Case C:} Closing step of $\{ b \}$ is free but not the last free closing step

{\bf Case D:} Closing step of $\{ b \}$ is the last free closing step

In all four subcases of Case 1 we derive below the estimate (cf.~\eqref{eq:C13.3})
\begin{equation}
\label{eq:C13.4}
0 \leq 1+ 2m(b) + L(b) - 2f(b)
\end{equation}
(or better) from which we already know that \eqref{eq:C13.2} follows.
In order to transfer the
above reasoning to all cases conveniently
we make precise what we mean by a visit of $\gamma$ at $b$:
It is a collection of consecutive steps $s$ up to $s+p$, $p \geq 1$, with $\gamma_s \neq b$,
$\gamma_{s+p+1} \neq b$ and $\gamma_i = b$ for all $s+1 \leq i \leq s+p$.

{\bf Case 1A:} The reasoning presented above holds also in the case that $\{ b \}$ is an edge
of higher multiplicity because $\{ b \}$ cannot be an open edge then.

{\bf Case 1B:}
%Inequalities \eqref{eq:C13.3} are replaced by
\begin{equation}
\label{eq:C13.5}
1 \leq \Lambda_{t_0}(b,\gamma) \leq 1 + 2m(b) + L(b) - 2f(b) \,.
\end{equation}
Indeed, in comparison with \eqref{eq:C13.3}, the additional first summand $1$ of the right hand side
accounts for the edge $\{ b \}$ that may be open after step $t_0$ is completed. The term $m(b)+1$
is replaced by $m(b)$ since the opening step of edge $\{ b \}$ does not initiate a visit at $b$.

{\bf Case 1C:}
%Inequalities \eqref{eq:C13.5} are replaced by
\begin{equation}
\label{eq:C13.6}
1 \leq \Lambda_{t_0}(b,\gamma) \leq 2m(b) + L(b) - 2(f(b)-1) \,.
\end{equation}
In comparison with \eqref{eq:C13.5} the first summand $1$ of the right hand side has vanished,
because edge $\{ b \}$ is already closed at step $t_0$. The term $f(b)$ must be replaced
by $f(b) - 1$ since the free closing step of edge $\{ b \}$ does not end a visit at $b$.

{\bf Case 1D:} Observe that in this case the visit at $b$ is not completed after step $t_0$
but after step $t_0+1$. Taking into account whether step $t_0+1$ is an opening step,
a closing step, or a step along an edge of higher multiplicity, we obtain in all three cases
\begin{equation}
\label{eq:C13.7}
2 \leq \Lambda_{t_0+1}(b,\gamma) \leq 2m(b) + L(b) - 2(f(b)-1) \,.
\end{equation}

For each of the subcases of Case 2 we may derive an inequality that improves on the estimate
for the corresponding subcase of Case 1 by $1$, i.e.~we have in the case that $b$ is the starting
point of  $\gamma$ always (cf.~\eqref{eq:C13.4})
\begin{equation}
\label{eq:C13.8}
0 \leq 2m(b) + L(b) - 2f(b)
\end{equation}
which implies again \eqref{eq:C13.2}. The reason for this is the same in all four subcases.
On the one hand one must increase the right hand side of the inequalities
\eqref{eq:C13.4} - \eqref{eq:C13.7} by $1$
to account for the possibility that the starting point $b$ is left for the first time by an opening step
(recall that the first part of $\gamma$ staying at $b$ is not considered a visit because the
step of arrival at $b$ is missing). On the other hand Case 2 does not allow for the possibility to start
a visit at $b$ by an innovative opening step which reduces the right hand side of the inequalities
\eqref{eq:C13.4} - \eqref{eq:C13.7} by $2\cdot1=2$.
\hfill \rule{0.5em}{0.5em}

\bigskip\bigskip

\noindent
\begin{tabular}{lcl}
\textbf{Werner Kirsch}&\quad& \texttt{werner.kirsch@fernuni-hagen.de}\\
\textbf{Thomas Kriecherbauer}&\quad& \texttt{thomas.kriecherbauer@uni-bayreuth.de}
\end{tabular}

\begin{thebibliography}{99}
\bibitem{Erdos17} O. Ajanki, L. Erd\H{o}s, T. Kr\"{u}ger: Universality for general Wigner-type matrices,
\textit{Probab. Theory Relat. Fields} \textbf{169}, 667--727 (2017).

\bibitem{Erdos18a} J. Alt, L. Erd\H{o}s, D. Schr\"{o}der: Correlated Random Matrices: Band Rigidity and Edge Universality,
arXiv: 1804.07744v3.

\bibitem{Aldous} D. Aldous: \textit{Exchangeability and related topics},
pp.~1-198 in: Lecture Notes in Mathematics 117, Springer (1985).

%\bibitem{AGZ} G. Anderson, A. Guionnet, O. Zeitouni: \textit{An introduction
%to random matrices}, Cambridge University Press (2010).


\bibitem{Arnold} L. Arnold: On the Asymptotic Distribution of the Eigenvalues of
Random Matrices, \textit{J. Math. Anal. Appl.}  \textbf{20}, 262--268 (1967).

\bibitem{Arnold71} L. Arnold: On Wigner's semicircle law for the eigenvalues of random matrices,
\textit{Z. Wahrscheinlichkeitstheorie verw. Geb.}  \textbf{19}, 191--198 (1971).



%\bibitem{Bai etal} Z. Bai, J. Silverstein: \textit{Spectral analysis of
%large dimensional random matrices}, Springer (2010).

\bibitem{BaiY} Z. Bai, Y. Yin: Necessary and sufficient conditions for almost sure convergence of the
largest eigenvalue of a Wigner matrix, Ann. Prob. 16, 1729--1741,  (1988).


\bibitem{BMP} L. Bogachev, S. Molchanov, L. Pastur: On the density of states of random band matrices
\textit{Math. Notes} \textbf{50}, 1232--1242 (1992).

%\bibitem{Baik et al} J. Baik, G. Ben Arous, S. P\'ech\'e: \textit{Phase transition of the largest eigenvalue for
%nonnull complex sample covariance matrices}, Ann. Prob. \textbf{33}, 1643--1697 (2005).

%\bibitem{Bryc et al} W. Bryc, A. Dembo, T. Jiang: Spectral measure
%of large random Hankel, Markov and Toeplitz matrices, \textit{Ann. Prob.} \textbf{34}, 1--38 (2006).

%\bibitem{Chatterjee} S. Chatterjee: \textit{A generalization of the
%Lindeberg principle}, Ann. Probab. \textbf{34,} 2061--2076 (2006).

%\bibitem{Ellis} R. Ellis: \textit{Entropy, large deviations, and statistical
%mechanics}, Springer 2006.

\bibitem{Catalano} R. Catalano: On weighted random band-matrices with dependencies, PhD thesis, FernUniversit\"{a}t Hagen, (2016).

\bibitem{Erdos18} L. Erd\H{o}s, P. M\"{u}hlbacher: Bounds on the norm of Wigner-type random matrices, arXiv: 1802.05175v1.

\bibitem{deFinetti} B. de Finetti: Funzione caratteristica di un fenomeno aleatorio, Atti della R.
Accademia Nazionale dei Lincei, Ser. 6, Memorie, Classe di Scienze Fisiche, Matematiche e
Naturali 4, 251--299 (1931).

\bibitem{deFinetti2} B. de Finetti: La prevision: ses lois logiques, ses sources subjectives, Annales de l'lnstitut Henri
Poincare, 7, 1--68  (1937).

\bibitem{Fleermann} M. Fleermann: The almost sure semicircle law for random band matrices with dependent entries, arXiv:1711.10196.


%\bibitem{FriesenLoewe1} O. Friesen, M. L\"{o}we: The Semicircle Law
%for Matrices with Independent Diagonals, \textit{J. Theoret. Probab.} \textbf{26},
%1084--1096 (2013).

%\bibitem{FriesenLoewe2} O. Friesen, M. L\"{o}we: A phase transition
%for the limiting spectral density of random matrices, \textit{Electron. J. Probab.}
%\textbf{18}, 1--17 (2013).

\bibitem{FuerediK} Z. F\H{u}redi, J. Koml\'{o}s:
The eigenvalues of random symmetric matrices,
\textit{Combinatorica} \textbf{1} no. 3, 233--241 (1981).

%\bibitem{Goetze Naumov Tikhomirov} F. G\"{o}tze, A. Naumov, A. Tikhomirov:
%Semicircle law for a class of random matrices with dependent entries, Preprint arXiv:1211.0389v2.
%Limit theorems for two classes of random matrices with dependent entries, \textit{Theory Probab. Appl.}  \textbf{59}, 23--39 (2015).

%\bibitem{Goetze Tikhomirov 1} F. G\"{o}tze, A. Tikhomirov: Limit
%theorems for spectra of random matrices with martingale structure, \textit{Theory
%Probab. Appl.} \textbf{51}, 42--64 (2007).

\bibitem{Grenander} U. Grenander: \textit{Probabilities on algebraic structures}, Wiley (1968).

\bibitem{HewittSavage} E. Hewitt, L. J. Savage: Symmetric measures on Cartesian products; Trans. Amer. Math. Soc. 80, 470--501 (1955).

\bibitem{HKW} W. Hochst\"{a}ttler, W. Kirsch, S. Warzel: Semicircle law for a matrix ensemble
with dependent entries, \textit{J. Theoret. Probab.} \textbf{29}, 1047--1068 (2016).

%\bibitem{HofmannS} K. Hofmann-Credner, M. Stolz: Wigner theorems for random matrices with dependent entries: ensembles associated to symmetric spaces and sample covariance matrices,
%\textit{Electron. Commun. Probab.} \textbf{13}, 401--414 (2008).

\bibitem{Kirsch} W. Kirsch: \textit{Moments in Probability}, book in
preparation, to appear at DeGruyter.

\bibitem{KK1} W. Kirsch, T. Kriecherbauer: Sixty years of moments for random matrices,
%Preprint  arXiv:1612.06725 to appear in:\\
F. Gesztesy et al (Edts.), \textit{Non-Linear Partial Differential Equations, Mathematical Physics, and Stochastic Analysis.
The Helge Holden Anniversary Volume}, EMS Congress Reports, 349--379 (2018).

\bibitem{KK2} W. Kirsch, T. Kriecherbauer: Semicircle Law for Generalized Curie-Weiss Matrix Ensembles at Subcritical Temperature,
to appear in: \textit{J. Theoret. Probab.}, arXiv:1703.05183.

%\bibitem{KK3} W. Kirsch, T. Kriecherbauer: Largest and second largest singular values of de Finetti random matrices; in preparation

%\bibitem{Latala} R. Lata\l a: \textit{Some estimates of norms of random
%matrices}, Proc. Amer. Math. Soc. \textbf{133}, 1273--1282 (2005).

%\bibitem{MarchenkoP} V. Marchenko, L. Pastur:
%Distribution of eigenvalues in certain sets of random matrices. \textit{Math. USSR-Sbornik} \textbf{1}, 457--483 (1967).

\bibitem{MPK} S. Molchanov, L. Pastur, A. Khorunzhii: Distribution of the eigenvalues of random band matrices in the limit of their infinite order
\textit{ Theoret. and Math. Phys.} \textbf{90}, 108--118 (1992).

\bibitem{LoeweS} M. L\"{o}we, K. Schubert: On the limiting spectral density of random
matrices filled with stochastic processes, to appear in: \textit{Random Operators and Stochastic Equations}, arXiv:1512.02498.

%\bibitem{Olver} F. Olver: \textit{Asymptotics and special functions},
%Academic Press (1974).

%\bibitem{Pastur72} L. Pastur: On the spectrum of random matrices, \textit{Theoret. and Math. Phys.} \textbf{10} no. 1, 67--74 (1972).

%\bibitem{Pastur} L. Pastur: Spectra of random selfadjoint operators,
%\textit{Russian Math. Surveys} \textbf{28}, 1--67 (1973).

\bibitem{Pastur Sherbina} L. Pastur, M. Sherbina: \textit{Eigenvalue
distribution of large random matrices,} Mathematical Surveys and Monographs
\textbf{171}, AMS (2011).

\bibitem{Stanley} R. Stanley:\textit{\ Catalan Numbers, }%
Cambridge University Press (2015).

%\bibitem{Schenker Schulz-Baldes} J. Schenker, H. Schulz-Baldes:
%Semicircle law and freeness for random matrices with symmetries or correlations,
%5\textit{Mathematical Research Letters} \textbf{12}, 531--542 (2005).

\bibitem{Tao} T. Tao: \textit{Topics in random matrix theory}, AMS (2012).

%\bibitem{Thompson} C. Thompson: \textit{Mathematical Statistical Mechanics},
%Princeton University Press (1979).

\bibitem{Wigner1} E. Wigner: Characteristic vectors of bordered matrices with infinite dimension,
\textit{Ann. Math.} \textbf{62}, 548--564 (1955).

\bibitem{Wigner2} E. Wigner: On the distribution of the roots of
certain symmetric matrices, \textit{Ann. Math.} \textbf{67}, 325--328 (1958).


\end{thebibliography}
\end{document}